\documentclass[11pt]{article}

\usepackage[T1]{fontenc}
\usepackage[utf8]{inputenc}
\usepackage[margin=1in]{geometry}
\usepackage{amsmath,amssymb,amsthm}
\usepackage{newtxtext,newtxmath}
\usepackage{bm,mathrsfs}
\usepackage{graphicx}
\usepackage{xcolor}
\definecolor{clemson-orange}{RGB}{234,106,32}
\definecolor{chicago-maroon}{RGB}{128,0,0}
\definecolor{northwestern-purple}{RGB}{82,0,99}
\definecolor{cornell-red}{RGB}{179,27,27}
\definecolor{sauder-green}{RGB}{171,180,0}
\definecolor{harveymudd-gold}{RGB}{178,139,51}
\definecolor{lawngreen}{RGB}{0,250,154}
\definecolor{gray}{RGB}{192,192,192}
\usepackage{enumitem,booktabs,multirow}
\usepackage{algorithm,algpseudocode}
\usepackage{tikz,subcaption}
\usepackage{microtype}
\usepackage[authoryear,round]{natbib}
\definecolor{linkblue}{RGB}{0,51,102}
\usepackage[colorlinks=true,allcolors=linkblue]{hyperref}
\usepackage[nameinlink]{cleveref}
\hypersetup{pdftitle={Tensor nuclear norm minimization with linear constraints}}

\theoremstyle{plain}
\newtheorem{theorem}{Theorem}
\newtheorem{lemma}{Lemma}
\newtheorem{proposition}{Proposition}

\theoremstyle{definition}
\newtheorem{definition}{Definition}
\newtheorem{example}{Example}
\theoremstyle{remark}
\newtheorem{remark}{Remark}

\def\R{\mathbb{R}}

\usetikzlibrary{arrows.meta,positioning}
\def\T{\boldsymbol{\mathcal{T}}}

\def\M{\boldsymbol{\mathcal{M}}}
\def\A{\boldsymbol{\mathcal{A}}}

\def\U{\mathbf{U}}
\def\K{\boldsymbol{\mathcal{K}}}
\def\B{\mathcal{B}}
\def\V{\mathbf{V} }
\def\D{\mathbf{D}}
\def\BT{\mathbf{T}}
\def\BZ{\mathbf{Z}}
\def\X{\mathcal{X}}

\def\P{\mathbf{P}}
\def\bQ{\mathbf{Q}}
\def\Y{\mathcal{Y}}

\def\S{\mathbb{S}}
\def\I{\mathbf{I}}
\def\Z{\boldsymbol{\mathcal{Z}}}

\def\Zero{\mathbf{0}}
\def\x{\mathbf{x}}
\def\y{\mathbf{y}}
\def\z{{\mathbf{z}}}
\def\p{\mathbf{p}}
\def\q{\mathbf{q}}

\def\tr{\mathrm{tr}}

\def\la{\langle}
\def\ra{\rangle}

\crefname{assumption}{Assumption}{Assumptions}
\crefname{lemma}{Lemma}{Lemmas}
\crefname{theorem}{Theorem}{Theorems}
\crefname{corollary}{Corollary}{Corollaries}
\crefname{proposition}{Proposition}{Propositions}
\crefname{condition}{Condition}{Conditions}
\crefname{claim}{Claim}{Claims}
\crefname{procedure}{Procedure}{Procedures}
\crefname{algorithm}{Algorithm}{Algorithms}
\crefname{figure}{Figure}{Figures}
\crefname{remark}{Remark}{Remarks}
\crefname{section}{Section}{Sections}
\crefname{procedure}{Procedure}{Procedures}
\crefname{example}{Example}{Examples}
\crefname{definition}{Definition}{Definitions}
\crefname{table}{Table}{Tables}
\crefname{equation}{}{}
\crefname{enumi}{}{}
\crefname{conjecture}{Conjecture}{Conjectures}
\crefname{step}{Step}{Steps}
\crefname{appendix}{Appendix}{Appendices}
\crefname{footnote}{Footnote}{Footnotes}

\title{Tensor nuclear norm minimization with linear constraints}
\author{
Renjie Chen\\
\small Department of Systems Engineering and Engineering Management,\\ \small The Chinese University of Hong Kong\\
\small \texttt{rchen@se.cuhk.edu.hk}
\and
Nanxi Zhang\\
\small Ivey Business School,\\ \small University of Western Ontario\\
\small \texttt{nzhang@ivey.ca}
\and
Bo Jiang\\
\small School of Information Management and Engineering,\\ \small Shanghai University of Finance and Economics.\\
\small \texttt{isyebojiang@gmail.com}
}
\date{}

\begin{document}
\maketitle
\begin{abstract}
We study the tensor nuclear norm minimization with general linear constraints, a fundamental model for low-rank tensor optimization that includes many practical applications in computer vision, recommendation systems, revenue management, transportation research. Although the tensor nuclear norm is a convex surrogate for the tensor rank, optimizing it remains computationally challenging because evaluating the nuclear norm itself is NP-hard. We propose a semidefinite programming framework that exposes a finite structure hidden in this problem, even though it appears to possess infinitely many constraints. In particular, we show that the tensor nuclear norm admits an exact SDP representation supported by a finite number of points on the unit sphere. Based on this representation, we develop two complementary solution approaches. The first replaces the unknown support set by a finite sphere discretization which yields a solution with a provable worst-case approximation guarantee. The second selects spherical points adaptively, returns an exact optimum upon finite termination, and otherwise converges asymptotically to an optimal solution. 
Numerical experiments on synthetic tensor completion instances and real highway traffic data demonstrate that the adaptive method achieves competitive, and often substantially improved, recovery accuracy relative to existing methods.
\end{abstract}


\section{Introduction}\label{sec: intro}
In this paper, we consider the tensor nuclear norm minimization with general linear constraints:
\begin{align}
      \min_{\T\in\R^{n_1\times n_2\times n_3}}\;\; &\|\T\|_*\nonumber\\
        \text{s.t.}\;\; &\mathscr{A}(\T)=\mathbf{b},\label{prob:tensor-comp-linear-constr}\tag{TNN-Min}
\end{align}
where $\mathscr{A}:\R^{n_1\times n_2\times n_3}\to \R^d$ is a given linear operator and $\mathbf{b}\in \R^d$ is a given vector.
The operator $\mathscr{A}$ is associated with $d$ tensors $\A_1,\dots,\A_d\in \R^{n_1\times n_2\times n_3}$ and the $\ell$th component of $\mathscr{A}(\T)$ is given by
$[\mathscr{A}(\T)]_\ell := \la \A_\ell, \T\ra$ for $\ell=1,\dots,d$.
One can think of each constraint as a way to observe some partial information of $\T$.
The objective function $\|\T\|_*$ is the nuclear norm of $\T$ (see \cref{sec:nuc-norm} for a formal definition).
For a clearer presentation, we state our results for third-order tensors, and our results can be easily extended to higher orders. Throughout this paper, we assume that the feasible set $\{\T\in\R^{n_1\times n_2\times n_3}:\mathscr{A}(\T)=\mathbf{b}\}$ is nonempty and that $\A_1,\dots,\A_d$ are linearly independent.

Problem \cref{prob:tensor-comp-linear-constr} captures a large family of tensor recovery tasks in which multi-dimensional data is only partially observed.
One %
typical case  is the \emph{tensor completion} problem, where the entries of tensor $\T$ are only known on an index set $\Omega\subsetneq[n_1]\times[n_2]\times[n_3]$, where $[n]:=\{1,\dots,n\}$ for a positive integer $n$.
The problem is commonly modeled as $\min\{\operatorname{rank}(\T):\T(\omega)=\M(\omega)~\forall\,\omega\in\Omega\}$,
where every $\M(\omega)$ is observed and $\operatorname{rank}(\T)$ is the CANDECOMP/PARAFAC (CP) rank of $\T$. %

Tensor completions arise in many practical applications. 
In computer vision, images and videos are naturally third-order arrays which may have corrupted or missing pixels \citep{wang2004compact,ref:liu2012tensor,5651225}.
In recommendation systems, user-item interactions under different contexts form a sparsely observed tensor \citep{5197422}.
In revenue management, \citet{farias2019learning} show that transaction data form a tensor whose slices correspond to different types of consumer behavior, with many entries missing due to promotions and seasonality. In transportation research, \citet{ref:chen2021low} model spatiotemporal highway traffic speeds (indexed by day, time-of-day, and sensor location) as a low-rank tensor with missing measurements.

From the optimization perspective, directly minimizing the CP rank is computationally difficult because the rank function is nonconvex and discontinuous, in addition to the NP-hardness of the CP rank \citep{hastad1990tensor,hillar2013most}.
For matrices, a straightforward approach is convexification: The matrix nuclear norm is the convex envelope of the matrix rank 
\citep{fazel2002matrix,recht2010guaranteed}, and minimizing it recovers low-rank matrices exactly under some mild conditions \citep{CandesRecht2009,candes2010power}.
\citet{yuan2016tensor} generalize this approach to tensors and replace the CP rank with the tensor nuclear norm.
They show that the vanilla nuclear-norm formulation better exploits the tensor structure and provably requires fewer observed entries for reliable completions compared with matricization approaches that flatten the tensor into a matrix.
It has been a common practice in the literature to use the tensor nuclear norm as a natural convex surrogate of the CP rank, making tensor completions a typical example of \cref{prob:tensor-comp-linear-constr}.

Although the tensor nuclear norm is the convex surrogate of the tensor rank, 
\cref{prob:tensor-comp-linear-constr} is still NP-hard to solve because \emph{computing} the tensor nuclear norm itself is in general NP-hard \citep{ref:friedland2018nuclear}. 
Even the tools for the computation of tensor nuclear norm are very limited, not to mention its optimization over certain constraints. To the best of our knowledge, no previous literature has studied the computation aspect of \cref{prob:tensor-comp-linear-constr}. This gap between the statistical appeal of the model and the computational difficulty motivates our work.

\subsection{Challenge, ideas and contributions}
\paragraph{Main challenges.} Although computing the tensor nuclear norm itself is extremely difficult, alternative formulations are available to characterize it. \cite{he2023approximation} recently proposed an SDP formulation with an infinite (in fact uncountable) number of constraints via its dual norm, the tensor spectral norm. In particular, let $\mathcal{\T}=(\BT_1,\dots,\BT_{n_1})\in\R^{n_1\times n_2 \times n_3}$ with $\BT_i\in\R^{n_2\times n_3}$ for $i\in[n_1]$. The nuclear norm of $\mathcal{\T}$ can be  formulated~\cite[Equation (3.6)]{he2023approximation} as 
\begin{align}
   \|\T\|_*= \max_{\|\Z\|_\sigma \leq 1} \la \T,\Z \ra\;= \max_{\Z}\left\{
        \langle \T,\Z\rangle
        :\begin{bmatrix}
        \I_{n_2} & \sum_{i=1}^{n_1}x_i\BZ_i \\
        \sum_{i=1}^{n_1}x_i\BZ_i^\top &\I_{n_3}
    \end{bmatrix}\succeq 0\;\forall\,\x\in \mathbb{B}^{n_1}\right\},\tag{D-SDP}\label{prob:nuclear-norm-SDP-formulation}
\end{align}
where $\mathbb{B}^{n_1}$ is the unit sphere in $\R^{n_1}$, $\x=(x_1,\dots,x_{n_1})^\top\in\R^{n_1}$, and $\mathcal{\Z}=(\BZ_1,\dots,\BZ_{n_1})\in\R^{n_1\times n_2 \times n_3}$ with $\BZ_i\in\R^{n_2\times n_3}$ for $i\in[n_1]$.
We call it D-SDP since this reformulation is derived from the dual norm of tensor nuclear norm.
By substituting this formulation into the original model \cref{prob:tensor-comp-linear-constr}, we arrive at a minimax problem
\begin{equation} \label{minimax}
  \min_{\mathscr{A}(\T)=\mathbf{b}}\max_{\BZ_1,\dots,\BZ_{n_1}} \left\{\sum_{i=1}^{n_1}\la \BT_i,\BZ_i\ra:\begin{bmatrix}
        \I_{n_2} & \sum_{i=1}^{n_1}x_i\BZ_i \\
        \sum_{i=1}^{n_1}x_i\BZ_i^\top &\I_{n_3}
    \end{bmatrix}\succeq 0\;\forall\,\x\in\mathbb{B}^{n_1}\right\},
\end{equation}
The inner problem, a convex optimization with infinitely many constraints, is called \emph{semi-infinite programming }(SIP) and known to be difficult to solve. 
Although numerous algorithms for minimax problems have been proposed in recent years, none are directly applicable to our model owing to the presence of infinitely many SDP constraints, which in turn renders projection onto the feasible region of \cref{minimax} challenging.
The core of our method in this paper is to overcome the infinity of \cref{minimax} by finding a finite set to approximate the constraint space while keeping the optimality or controlling the optimality loss.

\paragraph{Sketch of the idea.}

The infinity issue of \cref{minimax} originates from \cref{prob:nuclear-norm-SDP-formulation}, where every point on the unit sphere corresponds to a constraint.
For the sake of tractable computation, we need to find a way to render the ``infinity" of the problem to ``finity".
The question is whether finitely many sphere-indexed constraints
can preserve the optimal value for a given tensor, or approximate
it with a controlled error. We first show that exact preservation
of the optimal value is possible.
In particular, for every $\T\in\R^{n_1\times n_2\times n_3}$, there exists a finite support set $\Theta\subsetneq \mathbb{B}^{n_1}$ with $|\Theta|\leq n_1n_2n_3$ such that the following SDP with finite number of constraints has the same optimal value as \cref{prob:nuclear-norm-SDP-formulation}:
 \begin{align*}
        \|\T\|_* = \max_{\Z}\left\{
        \langle \T,\Z\rangle
        :\begin{bmatrix}
        \I_{n_2} & \sum_{i=1}^{n_1}x_i\BZ_i \\
        \sum_{i=1}^{n_1}x_i\BZ_i^\top &\I_{n_3}
    \end{bmatrix}\succeq 0\;\forall\,\x\in\Theta
        \right\}.
    \end{align*}
This result {\it theoretically} gives a finite SDP representation to tensor nuclear norm. It is worth mentioning that such result is 
derived from another SDP formulation of the tensor nuclear norm and we call it ``P-SDP" (see \cref{sec:nuc-norm}), as it is from the definition of the nuclear norm in \cref{tensor-nuclear-norm}, i.e., the primal side. We show that the strong duality holds for \cref{prob:nuclear-norm-SDP-formulation,vD-eq2}.

The crux of the above problem thus reduces to identifying the set \(\Theta\), which nevertheless remains a difficult task.
If this were not the case, the original problem would be solvable in polynomial time.
In this paper, we bypass the difficulty of finding the exact $\Theta$, 
instead, we try to find good approximations. Two approaches are developed. 
The first (\cref{sec:compute}) adopts the sphere
discretization idea in \cite{he2023approximation} and uses the so-called $\tau$-hitting set to replace the set $\Theta$. We show that the resulting approximation problem has worst-case theoretical guarantee of the original problem. The second approach (\cref{sec:dual}) selects sphere points adaptively, a similar idea to column generation for linear programming adopted in \cite{he2026nuclear} to compute the tensor nuclear norm. We prove that the adaptive procedure converges asymptotically to the optimal value of \cref{prob:tensor-comp-linear-constr}.

\paragraph{Contributions.}

We summarize our contributions as follows.
\begin{enumerate}[label=(C\arabic*),leftmargin=*]
    \item \emph{A finite SDP representation of the tensor nuclear norm.} \cref{thm:nuc-norm-finite-rep} represents the tensor nuclear norm as a finite SDP supported on at most $n_1n_2n_3$ points of the unit sphere. Our analysis (\cref{lemma:nuclear-rank-ub}) also provides an explicit upper bound on the tensor nuclear rank.\label{item:c1}
    \item \emph{A constant-factor approximation method.} Discretizing the sphere by a $\tau$-hitting set yields a $\tau$-approximate solution of \cref{prob:tensor-comp-linear-constr} (\cref{thm:approximation-ratio}). To the best of our knowledge, this is the first method with a provable worst-case guarantee that directly minimizes the tensor nuclear norm (rather than some computational surrogates) with linear constraints.\label{item:c2}
    \item \emph{An adaptive, asymptotically optimal algorithm.} \cref{alg:adaptive-H} constructs the point set adaptively and returns an exact optimum upon finite termination and an asymptotically optimal solution otherwise (\cref{thm:dual-app,thm:alg-converge,thm:dual-dis}).\label{item:c3}
    \item We perform numerical experiments of tensor completions by applying on our main model and the adaptive algorithm. Based on synthetic and highway traffic data, the results obviously outperform state-of-the-art methods.
\end{enumerate}

\subsection{Related literature}\label{subsec:literature}
\emph{Tensor completion via rank minimization and tractable surrogates.}
Low-rank tensor recovery is classically formulated as rank minimization. 
For matrices, the nuclear norm provides the canonical convex surrogate for the rank. It is the convex envelope of the rank function on the spectral-norm unit ball \citep{fazel2002matrix}, and nuclear-norm minimization admits exact-recovery guarantees under general linear measurements \citep{recht2010guaranteed},
in particular for matrix completions \citep{CandesRecht2009,candes2010power}; see \citet{ref:davenport2016overview} for an overview.

The tensor low-rank recovery is much more challenging. 
While one would ideally recover a low-rank tensor by minimizing its CP rank subject to observations, computing the CP rank is NP-hard \citep{hastad1990tensor,hillar2013most}. 
The tensor nuclear norm provides a natural convex surrogate for the CP rank \citep{ref:lim2013blind,ref:friedland2018nuclear}. \citet{yuan2016tensor} has established recovery guarantees for tensor completions by directly minimizing this nuclear norm, showing advantages over matricization-based approaches.

Computing the tensor nuclear norm itself is NP-hard \citep{ref:friedland2018nuclear} and this has triggered a broad literature on more tractable surrogates of tensor rank, including sums of nuclear norms of tensor unfoldings \citep{ref:liu2012tensor}, matricization-based convex relaxations \citep{pmlr-v32-mu14}, t-SVD-based tubal nuclear norms \citep{zhang2014novel,ref:zhang2016exact}, Kronecker-basis-representation sparsity measures \citep{8000407}, M-ranks of tensors~\citep{jiang2018low}, and slice-based learning models \citep{farias2019learning}; see also \citet{ref:chen2021low} for an application to spatiotemporal traffic completion, and 
\citet{shen2022smooth} for an application to video background/foreground separation with missing pixels.
These tractable surrogates do not provide approximation recovery guarantees relative to \cref{prob:tensor-comp-linear-constr}.
In contrast, in this paper, we study the vanilla tensor nuclear norm, and develop algorithms with performance guarantees relative to this objective; the aforementioned methods serve as numerical benchmarks in \cref{sec:numerical}.

\emph{Computation and approximation of tensor norms.}
Both the tensor spectral norm and nuclear norm are NP-hard to compute \citep{he2010approximation,ref:friedland2018nuclear}. Algorithms for homogeneous polynomial optimization and best rank-one approximation \citep{he2010approximation,so2011approximation,he2014probability,friedland2013best,anandkumar2017analyzing} provide spectral-norm approximation tools, including the subroutines used in \cref{alg:adaptive-H}. \citet{nie2017symmetric} gives the first method for computing the tensor nuclear norm for symmetric tensor, \citet{hu2025complexity} refine the complexity analysis for third-order tensors with one fixed dimension and provide a fully polynomial-time approximation scheme, and \citet{he2026nuclear} propose a framework to compute the tensor nuclear norm via duality. Closest to our fixed discretization, \citet{he2023approximation} approximate tensor spectral and nuclear norms using sphere coverings by $\tau$-hitting sets. We establish that a tensor-dependent finite set can represent the nuclear norm exactly and, more importantly, extend the computational framework from evaluating a given tensor's norm to minimizing it with general linear constraints.

\subsection{Organization and notation}
The remainder of the paper is organized as follows.
\cref{sec:nuc-norm} develops the finite SDP representation of the tensor nuclear norm (\cref{thm:nuc-norm-finite-rep}).
\cref{sec:compute} builds the hitting-set discretization and proves the approximation guarantee (\cref{thm:approximation-ratio}).
\cref{sec:dual} develops the dual perspective and the adaptive algorithm with its convergence analysis (\cref{thm:dual-dis,thm:dual-app,thm:alg-converge}).
\cref{sec:numerical} reports numerical experiments on synthetic and highway traffic data, and \cref{sec:conclusion} concludes.
Proofs not given in the main text appear in the appendix.

Throughout, we use boldface calligraphic letters (e.g., $\T,\A,\Z$) for tensors, capital Greek letters (e.g., $\Theta,\Upsilon$) for sets, boldface capital letters (e.g., $\U,\V,\mathbf{X}$) for matrices, and boldface lowercase letters (e.g., $\x,\y,\bm\lambda$) for vectors.
We reserve the index $i\in[m]$ for discretization points, $\ell\in[d]$ for linear constraints, $r\in[R]$ for rank-one terms, $t$ for algorithm iterations, and $j\in[n_1]$ for coordinates.
$\mathbb{B}^{n_1}:=\{\x\in\R^{n_1}:\|\x\|=1\}$ denotes the unit sphere in $\R^{n_1}$.

\section{ SDP Reformulation of nuclear norm 
}
\label{sec:nuc-norm} 

In this section, we propose new formulations to compute the tensor nuclear norm. They form the basis for the analysis of the original problem \cref{prob:tensor-comp-linear-constr} in later sections.
To start with, we give a formal definition of tensor nuclear norm, tensor nuclear decomposition and tensor nuclear rank.

\begin{definition}[Tensor nuclear norm]
    The nuclear norm of a tensor $\mathcal{T} \in \mathbb{R}^{n_1 \times n_2 \times \cdots \times n_d}$ is 
\begin{align}\label{tensor-nuclear-norm}
    \|\T\|_*:=\min_{\lambda_r,\x_r, R} \left\{\sum_{r=1}^R|\lambda_r|: \T = \sum_{r=1}^R \lambda_r \x_1^{(r)} \otimes \x_2^{(r)} \otimes\cdots\otimes \x_d^{(r)}, R \in \mathbb{N},~\|\x_{\ell}^{(r)}\|=1 \forall\,\ell\in[d], r\in[R] \right\},
\end{align}
where %
$\otimes$ denotes the vector outer product.
\hfill $\triangleleft$
\end{definition}

\begin{definition}[Tensor nuclear decomposition]
    If a tensor $\mathcal{T} \in \mathbb{R}^{n_1 \times n_2 \times \cdots \times n_d}$, written as a rank-one decomposition
    \begin{align}
     &\T \;=\; \sum_{r=1}^R \lambda_r \x_1^{(r)} \otimes \x_2^{(r)} \otimes \cdots \otimes\x_d^{(r)},~\label{def:nuclear-decomp}
\end{align}
satisfies that $\|\T\|_* = \sum_{r=1}^R |\lambda_r|$, then \cref{def:nuclear-decomp} is called a nuclear decomposition of $\T$.
\hfill $\triangleleft$
\end{definition}

\begin{definition}[Tensor nuclear rank]
    The nuclear rank of a tensor $\T \in \R^{n_1 \times n_2 \times \cdots \times n_d}$ is defined by
        \begin{align}
        \label{def:nuclear-rank}
        \operatorname{rank}_*(\T) := \min\left\{R\in\mathbb{N}: \T = \sum_{r=1}^R \lambda_r \x_1^{(r)} \otimes\cdots\otimes \x_d^{(r)}, \|\T\|_*= \sum_{r=1}^R|\lambda_r|, ~\|\x_{\ell}^{(r)}\|=1 \forall\,\ell\in[d], r\in[R]\right\}.
    \end{align}
    \hfill $\triangleleft$
\end{definition}

In the rest of this paper, we establish our results for the case $d=3$. All these results can be easily extended to higher orders.

We first provide an upper bound on the nuclear rank for third-order tensors.
The proof of \cref{lemma:nuclear-rank-ub} is given in \cref{app:nuclear-rank-ub}.
\begin{lemma}\label{lemma:nuclear-rank-ub}
    $\operatorname{rank}_*(\T) \leq n_1n_2n_3$ for any $\T \in \R^{n_1\times n_2\times n_3}$. 
\end{lemma}

\begin{remark}
This bound is also implied by \citet[Proposition~4.3]{ref:friedland2018nuclear}. We include a short proof based on Carath\'{e}odory's theorem to make the finite-support argument explicit. The same reasoning applies to tensors of higher order.
\end{remark}

To obtain an SDP, we keep the first factor of each rank-one term and combine the remaining two factors into a matrix. The next theorem shows that this grouping preserves the nuclear norm when the first-mode directions are chosen appropriately.
\begin{theorem}\label{thm:nuc-norm-finite-rep}
    For any $\T \in \R^{n_1\times n_2\times n_3}$, there exists a set $\Theta\subseteq \mathbb{B}^{n_1}$ such that $|\Theta| \leq n_1n_2n_3$ and
    \begin{align}
        \|\T\|_* = \min_{\bm\Lambda} \left\{\frac{1}{2}\sum_{\x\in\Theta} \tr(\bm\Lambda(\x)): \T=\sum_{\x\in\Theta}\x\otimes\bm\Gamma(\x),\bm\Lambda(\x) = \begin{bmatrix}
        \bm\Lambda^{(1)}(\x) &\bm\Gamma(\x)\\
        \bm\Gamma^\top(\x) &\bm\Lambda^{(2)}(\x)
    \end{bmatrix}\succeq \bm 0~~\text{for all $\x\in \Theta$}
    \right\}\label{vD-eq1}
    \end{align}
\end{theorem}
The relation in \cref{vD-eq1} reduces the computation of a tensor nuclear norm to the addition of a series of matrix traces and these matrices are all positive semi-definite (PSD) matrices.
Note that when $\Theta$ is given, the minimization problem in \cref{vD-eq1} is a semi-definite program, which can be computed by a solver. 
We study \cref{prob:tensor-comp-linear-constr} based on the nuclear representation in \cref{thm:nuc-norm-finite-rep} in this paper.

\cref{thm:nuc-norm-finite-rep} essentially  builds the connection between tensor nuclear norm to matrix trace and a series of PSD constraints.
To prove \cref{thm:nuc-norm-finite-rep}, we use tensor nuclear norm as a bridge. In the next lemma, we first show that there is a connection between matrix nuclear norm and matrix trace. 
\begin{lemma}\label{lemma:matrix-nuclear-norm} Given any $\bm\Gamma\in\R^{n_1\times n_2}$, let $\bm\Gamma=\P^\top \D \bQ$ be its SVD and $\omega_1,\dots,\omega_{n_1}$ be its singular values, i.e., $\D=[\mathrm{diag}(\omega_1,\dots,\omega_{n_1},0,\dots,0)]$. Then the optimization problem parameterized by $\bm\Gamma$
\begin{equation}\label{eq:inf-trace-sum}
    \min_{\U,\V}\left\{\frac{1}{2}\left(\mathrm{tr}(\U)+\mathrm{tr}(\V)\right):\begin{bmatrix}
        \U &\bm\Gamma\\
        \bm\Gamma^\top &\V
    \end{bmatrix}\succeq \bm 0,~\U\in\R^{n_1\times n_1}, \V\in\R^{n_2\times n_2}\right\}
\end{equation}
is solved by
\begin{equation}\label{UV-def}
      \U^\star(\bm\Gamma) = \mathbf{P}^\top\mathrm{diag}(\omega_1,\dots,\omega_{n_1})\mathbf{P}, ~\V^\star(\bm\Gamma) = \mathbf{Q}^\top\mathrm{diag}(\omega_1,\dots,\omega_{n_1},0,\dots,0)\mathbf{Q}
\end{equation}
and $\text{val}\cref{eq:inf-trace-sum}= \|\bm\Gamma\|_*$.
\end{lemma}

\cref{lemma:matrix-nuclear-norm} shows that the matrix nuclear norm can be calculated by problem \cref{eq:inf-trace-sum} with PSD constraints.
The proof of \cref{lemma:matrix-nuclear-norm} is deferred to \cref{app:matrix-nuclear-norm}.
With \cref{lemma:matrix-nuclear-norm}, we are ready to prove \cref{thm:nuc-norm-finite-rep}.
\begin{proof}[Proof of \cref{thm:nuc-norm-finite-rep}]
    The proof idea is as follows.
    For any given $\Theta=\{\tilde{\x}^{(1)},\dots,\tilde{\x}^{(R)}\}$, we use the following problem as an intermediary:
    \begin{align}
        \min_{W^{(r)}}\left\{\sum_{r=1}^R\|\bm W^{(r)}\|_*: \T = \sum_{r=1}^R \tilde{\x}^{(r)} \otimes {\bm W}^{(r)},~\bm W^{(r)}\in \R^{n_2\times n_3}\right\}.\label{eq:nuclear-norm-no-rank}
    \end{align}
    We first prove 
    that  $\text{val}\cref{vD-eq1}=\text{val}\cref{eq:nuclear-norm-no-rank}$ for any given $\Theta$. Then we prove that there exists $\Theta$ such that $\|\T\|_*=\text{val}\cref{eq:nuclear-norm-no-rank}$ and the result follows.
    
    Now we prove $\text{val}\cref{vD-eq1}=\text{val}\cref{eq:nuclear-norm-no-rank}$ for any given $\Theta$.
    To see this, we note that by \cref{lemma:matrix-nuclear-norm}, for any matrix $\bm W\in \R^{p\times q}$, it holds that 
    \[
        \|\bm W\|_* = \min_{\U,\V}\left\{\frac{1}{2}\left(\mathrm{tr}(\U)+\mathrm{tr}(\V)\right):\begin{bmatrix}
        \U &\bm W\\
        \bm W^\top &\V
    \end{bmatrix}\succeq \bm 0,~\U\in\R^{p\times p}, \V\in\R^{q\times q}\right\}.
    \]
    Then \cref{eq:nuclear-norm-no-rank} is equivalent to
    \begin{align}
        \min_{\mathbf{U},\mathbf{V},\bm W^{(r)}}\left\{\frac{1}{2}\sum_{r=1}^R\left(\tr(\mathbf{U}^{(r)}) + \tr(\mathbf{V}^{(r)})\right): \T = \sum_{r=1}^R \tilde{\x}^{(r)} \otimes {\bm W}^{(r)},\begin{bmatrix}
        \U^{(r)} &\bm W^{(r)}\\
        {\bm W^{(r)}}^\top &\V^{(r)}
    \end{bmatrix}\succeq \bm 0\right\}.\label{eq:nuclear-norm-no-rank-sdp}
    \end{align}
    Note that \cref{eq:nuclear-norm-no-rank-sdp} is equivalent to \cref{vD-eq1}. Thus $\text{val}\cref{vD-eq1}=\text{val}\cref{eq:nuclear-norm-no-rank-sdp}=\text{val}\cref{eq:nuclear-norm-no-rank}$.
    
    Now we prove that there exists $\Theta$ such that $\|\T\|_*=\text{val}\cref{eq:nuclear-norm-no-rank}$. By \cref{lemma:nuclear-rank-ub} and the definition of nuclear norm, it holds that 
    \begingroup\small
\begin{align}
        \|\T\|_* \;=\; &\min_{\bm\lambda_r,\x_1^{(r)},\x_2^{(r)},\x_3^{(r)},R}\left\{\sum_{r=1}^R|\lambda_r|: \T = \sum_{r=1}^R \lambda_r \x_1^{(r)} \otimes \x_2^{(r)} \otimes \x_3^{(r)},~ 
    \|\x_{d}^{(r)}\|=1,~d\in\{1,2,3\}, R \in \mathbb{N},R\leq n_1n_2n_3\right\}\nonumber\\
    \;=\; & \min_{\x_1^{(r)},\bm W^{(r)},R}\left\{\sum_{r=1}^R\|\bm W^{(r)}\|_*: \T = \sum_{r=1}^R \x_1^{(r)} \otimes {\bm W}^{(r)},~ 
    \|\x_{1}^{(r)}\|=1, \operatorname{rank}(\bm W^{(r)})=1, R \in \mathbb{N},R\leq n_1n_2n_3\right\},\label{eq:nuclear-norm-rank-1}
    \end{align}
\endgroup
    where in the second equality, we use the fact that $\|\lambda_r \x_2^{(r)}\otimes \x_3^{(r)}\|_* = |\lambda_r|$. 
    Let $\{{\x_1^*}^{(r)},r=1,\dots,R\}$ be a set of optimal solution to \cref{eq:nuclear-norm-rank-1}. Let $\Theta=\{{\x_1^*}^{(1)},\dots,{\x_1^*}^{(R)}\}$.
    In the following, we show that under $\Theta$ constructed in this way, $\text{val}\cref{eq:nuclear-norm-rank-1}=\text{val}\cref{eq:nuclear-norm-no-rank}$.
    It is straightforward that $\text{val}\cref{eq:nuclear-norm-no-rank}\leq \text{val}\cref{eq:nuclear-norm-rank-1}$, since the feasible region of \cref{eq:nuclear-norm-rank-1} is a subset of that of \cref{eq:nuclear-norm-no-rank}. Now we show that $\text{val}\cref{eq:nuclear-norm-no-rank}\geq \text{val}\cref{eq:nuclear-norm-rank-1}$. First note that $\text{val}\cref{eq:nuclear-norm-rank-1}=\text{val}\cref{tensor-nuclear-norm}$ according to the definition of tensor nuclear norm. Next, we will show that for every feasible solution of \cref{eq:nuclear-norm-no-rank}, we can construct a feasible solution of \cref{tensor-nuclear-norm}, and they share the same objective value. 
    
    Let $\bm W^{(r)}, r=1,\dots,R$ be a feasible solution to \cref{eq:nuclear-norm-no-rank}. Then it holds that 
    $\T = \sum_{r=1}^R {\x_1^*}^{(r)} \otimes {\bm W}^{(r)}$. 
    For each $r$, take a SVD decomposition of $\bm W^{(r)}$, i.e., 
    \[
        \bm W^{(r)} = \sum_{j=1}^{n_2} \sigma_{r,j} \mathbf{u}_{r,j} \otimes \mathbf{v}_{r,j},~\sigma_{r,j} \geq 0,~ \|\mathbf{u}_{r,j}\| = \|\mathbf{v}_{r,j}\|=1,
    \]
    where $\|\bm W^{(r)}\|_*=\sum_{j=1}^{n_2} \sigma_{r,j}$. Then it holds that 
    \[
        \T = \sum_{r=1}^R \sum_{j=1}^{n_2} \sigma_{r,j} {\x_1^*}^{(r)} \otimes \mathbf{u}_{r,j} \otimes \mathbf{v}_{r,j}.
    \]
    Note that $\{{\x_1^*}^{(r)}, \mathbf{u}_{r,j}, \mathbf{v}_{r,j}\}$ form a feasible solution to the minimization problem in \cref{tensor-nuclear-norm} and the objective value of \cref{tensor-nuclear-norm} equals $\sum_{r=1}^R \sum_{j=1}^{n_2} \sigma_{r,j}$ under this solution. 
    Because $\sum_{r=1}^R \|\bm W^{(r)}\|_* = \sum_{r=1}^R \sum_{j=1}^{n_2} \sigma_{r,j}$, we prove that for every feasible solution of \cref{eq:nuclear-norm-no-rank}, we can construct a feasible solution of \cref{tensor-nuclear-norm}, and they share the same objective value. 
    
\end{proof}

According to the proof of \cref{thm:nuc-norm-finite-rep}, the tensor nuclear norm can be represented by the following problem where $\Theta$ is a decision variable:
    \begingroup\small
\begin{align}
        \|\T\|_* = \min_{\bm\Lambda,\Theta\subset \mathbb{B}^{n_1}} \left\{\frac{1}{2}\sum_{\x\in\Theta} \tr(\bm\Lambda(\x)): \T=\sum_{\x\in\Theta}\x\otimes\bm\Gamma(\x),\bm\Lambda(\x) = \begin{bmatrix}
        \bm\Lambda^{(1)}(\x) &\bm\Gamma(\x)\\
        \bm\Gamma^\top(\x) &\bm\Lambda^{(2)}(\x)
    \end{bmatrix}\succeq \bm 0~~\text{for $\x\in\Theta$}, |\Theta| \leq n_1n_2n_3.
    \right\}\tag{P-SDP}\label{vD-eq2}
    \end{align}
\endgroup
    \hfill

\cref{thm:nuc-norm-finite-rep} shows that the tensor nuclear norm can be obtained by solving \cref{vD-eq1}, which has a linear objective and a finite set of constraints with supports on $\Theta$. 
However, finding the specific subset $\Theta$ in \cref{thm:nuc-norm-finite-rep} is not a trivial task. 
Therefore, in this paper, we make a compromise and approximate~\cref{vD-eq1} by assuming $\Theta$ is given to be $\{\x_1,\dots,\x_m\}$ and consider the following problem
\begin{align}
    \min_{\bm\Lambda_1,\dots,\bm\Lambda_m}\;\; &\frac{1}{2}\sum_{i=1}^{m} \tr\big(\bm\Lambda_i\big)\nonumber\\
    \text{s.t.} \;\;&\bm\Lambda_i=\begin{bmatrix}
\bm\Lambda^{(1)}_i &\bm\Gamma_i\\
\bm\Gamma_i^\top&\bm\Lambda^{(2)}_i
\end{bmatrix}\succeq \bm 0~~\forall\, i=1,\dots,m,\nonumber\\
    &\T=\sum_{i=1}^m \x_i\otimes\bm\Gamma_i.\label{prob:vD-app}
\end{align}
Note that $\text{val}\cref{prob:vD-app}$ may not be the same as $\text{val}\cref{vD-eq1}$, and the approximation performance heavily relies on how $\{\x_1,\dots,\x_m\}$ are selected.
We show two ways of construction in this work and show their performance in later sections.

\paragraph{Representation through the dual norm.}
{We now derive \cref{prob:nuclear-norm-SDP-formulation} and explain why its infinitely many constraints can also be reduced to a finite set for each given tensor.}
\begin{definition}[Tensor spectral norm]
The \emph{spectral norm} of a tensor $\Z\in\R^{n_1\times\cdots\times n_d}$ is
\[
\|\Z\|_\sigma:=\max_{\|\x_1\|=\cdots=\|\x_d\|=1}
\la\Z,\x_1\otimes\cdots\otimes\x_d\ra.
\]
\end{definition}
{The spectral norm is the dual norm to the nuclear norm, i.e., $\|\T\|_\sigma=\max_{\|\Z\|_*\leq1}\la\T,\Z\ra$ and $\|\T\|_*=\max_{\|\Z\|_\sigma\leq1}\la\T,\Z\ra$ \citep[Lemma 2.3]{lipartition}.

Given a third-order tensor $\Z=(\BZ_1,\dots,\BZ_{n_1})\in\R^{n_1\times n_2\times n_3}$ with slices $\BZ_i\in\R^{n_2\times n_3}$ and a vector $\x=(x_1,\dots,x_{n_1})^\top$, we define the contraction $\Z(\x):=\sum_{i=1}^{n_1}x_i\BZ_i\in\R^{n_2\times n_3}$. Maximizing over the other two unit vectors gives \citep{he2023approximation}}
\begin{equation}\label{eq:tensor-norm-matrix-norm}
\|\Z\|_\sigma=\max_{\|\x\|=1}\max_{\|\y\|=\|\z\|=1}\la\Z,\x\otimes\y\otimes\z\ra
=\max_{\|\x\|=1}\|\Z(\x)\|_\sigma.
\end{equation}
By the Schur complement,
\begin{equation}\label{eq:schur-spec}
\|\Z(\x)\|_\sigma\leq1\ \Longleftrightarrow\ 
\I_{n_3}\succeq\Z(\x)^\top\Z(\x)\ \Longleftrightarrow\ 
\begin{bmatrix}\I_{n_2}&\Z(\x)\\\Z(\x)^\top&\I_{n_3}\end{bmatrix}\succeq\bm0.
\end{equation}
{Thus requiring $\|\Z\|_\sigma\leq1$ imposes the semidefinite constraint in \cref{eq:schur-spec} at every sphere point. 
The nuclear norm calculation problem thus becomes
\begin{align*}
     \|\T\|_*= \max_{\Z}\left\{
        \langle \T,\Z\rangle
        :\begin{bmatrix}
        \I_{n_2} & \Z(\x) \\
        \Z(\x)^\top &\I_{n_3}
    \end{bmatrix}\succeq 0\;\forall\,\x\in \mathbb{B}^{n_1}\right\},
\end{align*}
which is \cref{prob:nuclear-norm-SDP-formulation}.
The following theorem shows that a tensor-dependent finite subset suffices to preserve its optimal value. Its proof is in \cref{app:nuc-norm-finite-rep-2}.}
\begin{theorem}\label{thm:nuc-norm-finite-rep-2}
For any tensor $\T$, there exists a set $\Theta\subseteq\mathbb B^{n_1}$ such that $|\Theta|\leq n_1n_2n_3$ and
\begin{align}
\|\T\|_* =\max_{\Z}\left\{\la\T,\Z\ra:
\begin{bmatrix}\I_{n_2}&\Z(\x)\\\Z(\x)^\top&\I_{n_3}\end{bmatrix}\succeq0
\quad\forall\x\in\Theta\right\}.\label{prob:nn-2}
\end{align}
\end{theorem}

\section{A fixed approximation ratio method}\label{sec:compute}
In this section, we develop an approximation method to \cref{prob:tensor-comp-linear-constr}.
We first formulate \cref{prob:tensor-comp-linear-constr} using the tensor nuclear norm representation in \cref{thm:nuc-norm-finite-rep}. 
Then we develop a method to find an approximate solution for the problem.

\subsection{A new formulation of \texorpdfstring{\cref{prob:tensor-comp-linear-constr}}{(\ref*{prob:tensor-comp-linear-constr})}}
Plugging \cref{vD-eq2} into~\cref{prob:tensor-comp-linear-constr} gives the following equivalent formulation of \cref{prob:tensor-comp-linear-constr}:
{
\renewcommand{\theequation}{PRIMAL}
\begin{subequations}\label{prob:GenLC-finite-rep}
\renewcommand{\theequation}{\theparentequation-\alph{equation}}
\begin{align}
    \min_{\T,~\Theta\subseteq \mathbb{B}^{n_1},~\bm\Lambda}\;\; &\frac{1}{2}\sum_{\x\in\Theta} \tr\big(\bm\Lambda(\x)\big)\\
    \text{s.t.} \;\;&\bm\Lambda(\x)=\begin{bmatrix}
        \bm\Lambda^{(1)}(\x) &\bm\Gamma(\x)\\
        \bm\Gamma^\top(\x) &\bm\Lambda^{(2)}(\x)
    \end{bmatrix}\succeq \bm 0~~\text{for all }\x\in\Theta,\\
    &\T=\sum_{\x\in\Theta} \x\otimes\bm\Gamma(\x),\label{prob:GenLC-finite-rep-decomp}\\
    &\mathscr{A}(\T)=\mathbf{b}.
\end{align}
\end{subequations}
}

Note that when the set $\Theta$ is given as a finite set, \cref{prob:GenLC-finite-rep} is a standard SDP.
The remaining difficulty is that the support set $\Theta\subset \mathbb{B}^{n_1}$ in~\cref{prob:GenLC-finite-rep} is a \emph{decision variable} and the finite subset of $\mathbb{B}^{n_1}$ that supports the optimal solution is unknown.
This motivates the following approximation strategy: instead of optimizing over the unknown $\Theta$, we fix in advance a finite set $\Upsilon=\{\x_1,\dots,\x_m\}\subseteq \mathbb{B}^{n_1}$ and restrict the support of $\bm\Lambda$ to 
$\Upsilon$.
Replacing $\Theta$ by $\Upsilon$ in~\cref{prob:GenLC-finite-rep} yields the following approximation of~\cref{prob:tensor-comp-linear-constr}:
{
\renewcommand{\theequation}{PRIMAL(D)}
\begin{subequations}
    \label{prob:primal-linear-constr-discret}
    \renewcommand{\theequation}{\theparentequation-\alph{equation}}
\begin{align}
    \min_{\T,~\bm\Lambda_1,\dots,\bm\Lambda_m}\;\; &\frac{1}{2}\sum_{i=1}^{m} \tr\big(\bm\Lambda_i\big)\\
        \text{s.t.} \;\;&\bm\Lambda_i=\begin{bmatrix}
    \bm\Lambda^{(1)}_i &\bm\Gamma_i\\
    \bm\Gamma_i^\top&\bm\Lambda^{(2)}_i
    \end{bmatrix}\succeq 0~~i \in \{1,\dots,m\},\\
        &\T=\sum_{i=1}^m \x_i\otimes\bm\Gamma_i,\\
   &\mathscr{A}(\T)=\mathbf{b}.
    \end{align}
\end{subequations}
}
Note that \cref{prob:primal-linear-constr-discret} has finite semi-definite constraints and linear constraints, and is solvable by existing SDP solvers.
Since the optimal support $\Theta$ may not be a subset of $\Upsilon$, problem~\cref{prob:primal-linear-constr-discret} is just an approximation of and in general not equivalent to~\cref{prob:GenLC-finite-rep}, and the approximation quality depends on how the points in $\Upsilon$  are chosen.
In the following, we first show that \cref{prob:primal-linear-constr-discret} is well-defined, based on which we study the approximation performance of \cref{prob:primal-linear-constr-discret}.
The proof of \cref{prop:reg-primal-dis} is given in \cref{app:reg-primal-dis}.

\begin{proposition}\label{prop:reg-primal-dis}
    If $\operatorname{span}(\Upsilon)=\R^{n_1}$, then \cref{prob:primal-linear-constr-discret} admits an optimal solution.
\end{proposition}

\subsection{A method with constant approximation bound}

By definition, points in $\Upsilon$ are all on the unit sphere. To guarantee an accurate approximation, one intuition is to let the set $\Upsilon$ pick ``representative points" on the unit sphere $\mathbb{B}^{n_1}$.
For this purpose, we use the sphere-covering idea first proposed by \cite{he2023approximation}.

The idea is based on discretization over the unit sphere: For any $\x\in\mathbb{B}^{n_1}$ and $\tau \in (0,1]$, we define $\B^{n_1}(\x, \tau):=\{\y\in\mathbb{B}^{n_1}: \y^\top \x \geq \tau\}$ to be a closed spherical cap with the angular radius $\arccos \tau$.
An illustration of $\B^{3}(\x, \tau)$ is the shaded area in \cref{fig:grid-on-sphere}. 
For any vector in $\y\in \B^n(\x,\tau)$, the angle between $\y$ and $\x$ is less than $\arccos\tau$ and we use $\x$ as an ``representative vector" for all vectors in $\B^{n_1}(\x,\tau)$.
Therefore $\tau$ can be viewed as a tolerance for approximation error.
We first define the concept of {\it $\tau$-hitting set}.
\begin{definition}[$\tau$-hitting Set]\label{def:tau-hitting-set}
    A set $\Upsilon= \{\x_i\in\mathbb{B}^{n_1}, i=1,2,\dots,m\}$ is called a $\tau$-hitting set with cardinality $m$ if $\bigcup_{i=1}^m \B^{n_1}(\x_i, \tau)=\mathbb{B}^{n_1}$, i.e., the $m$ spherical caps covering the unit sphere. 
\end{definition}

Obviously, the smaller $\arccos\tau$ is, the more accurate the approximation $\Upsilon$ is to $\mathbb{B}^{n_1}$ and the larger $m$ is.
We give one way of constructing a $\tau$-hitting set in the following example.
For more ways of the construction of hitting sets and proofs of hitting ratios, see Section 2 of \cite{he2023approximation}.
\begin{example}\label{exmp:grid-on-sphere}
    For any $\x=(x_1,x_2,\dots,x_n)^\top\in\mathbb{B}^n$, let its spherical coordinates be denoted by $(\varphi_1,\varphi_2,\dots,\varphi_{n-1})$ with $\varphi_1,\varphi_2,\dots,\varphi_{n-2}\in [0,\pi]$ and $\varphi_{n-1}\in [0,2\pi]$ such that 
    \begin{align*}
        x_j&=\left(\prod_{k=1}^{j-1}\sin\varphi_k\right)\cos\varphi_j,
        \quad j=1,\dots,n-1,\qquad
        x_n=\prod_{k=1}^{n-1}\sin\varphi_k.
    \end{align*}

   \begin{figure}
    \centering
            \begin{tikzpicture}[scale=2.8,
        pt/.style={circle,fill=black,inner sep=1.0pt},
        bpt/.style={circle,draw=black!55,fill=white,inner sep=0.9pt,line width=0.3pt}]

    \def\az{25}\def\el{28}

    \newcommand\PRJ[2]{%
      \pgfmathsetmacro{\X}{sin(#1)*cos(#2)}%
      \pgfmathsetmacro{\Y}{sin(#1)*sin(#2)}%
      \pgfmathsetmacro{\Zc}{cos(#1)}%
      \pgfmathsetmacro{\resx}{ -\X*sin(\az) + \Y*cos(\az) }%
      \pgfmathsetmacro{\resy}{ -\X*cos(\az)*sin(\el) - \Y*sin(\az)*sin(\el) + \Zc*cos(\el) }%
      \pgfmathsetmacro{\resd}{  \X*cos(\az)*cos(\el) + \Y*sin(\az)*cos(\el) + \Zc*sin(\el) }%
    }
    \newcommand\PRJV[3]{%
      \pgfmathsetmacro{\vx}{ -(#1)*sin(\az) + (#2)*cos(\az) }%
      \pgfmathsetmacro{\vy}{ -(#1)*cos(\az)*sin(\el) - (#2)*sin(\az)*sin(\el) + (#3)*cos(\el) }%
      \pgfmathsetmacro{\vd}{  (#1)*cos(\az)*cos(\el) + (#2)*sin(\az)*cos(\el) + (#3)*sin(\el) }%
    }

    \draw[thick] (0,0) circle (1);

    \foreach \th in {36,72,108,144}{
      \foreach \a in {0,5,...,355}{
        \pgfmathsetmacro{\aa}{\a+5}
        \PRJ{\th}{\a}\pgfmathsetmacro{\axx}{\resx}\pgfmathsetmacro{\ayy}{\resy}\pgfmathsetmacro{\ad}{\resd}
        \PRJ{\th}{\aa}\pgfmathsetmacro{\bxx}{\resx}\pgfmathsetmacro{\byy}{\resy}\pgfmathsetmacro{\bd}{\resd}
        \pgfmathparse{(\ad>=0 && \bd>=0)?1:0}
        \ifnum\pgfmathresult=1 \draw[gray!55,thin](\axx,\ayy)--(\bxx,\byy);
        \else \draw[gray!22,thin](\axx,\ayy)--(\bxx,\byy);\fi
      }
    }
    \foreach \ph in {0,36,72,108,144}{
      \foreach \t in {0,5,...,355}{
        \pgfmathsetmacro{\tt}{\t+5}
        \PRJ{\t}{\ph}\pgfmathsetmacro{\axx}{\resx}\pgfmathsetmacro{\ayy}{\resy}\pgfmathsetmacro{\ad}{\resd}
        \PRJ{\tt}{\ph}\pgfmathsetmacro{\bxx}{\resx}\pgfmathsetmacro{\byy}{\resy}\pgfmathsetmacro{\bd}{\resd}
        \pgfmathparse{(\ad>=0 && \bd>=0)?1:0}
        \ifnum\pgfmathresult=1 \draw[gray!55,thin](\axx,\ayy)--(\bxx,\byy);
        \else \draw[gray!22,thin](\axx,\ayy)--(\bxx,\byy);\fi
      }
    }

    \foreach \th in {36,72,108,144}{
      \foreach \ph in {0,36,72,108,144,180,216,252,288,324}{
        \PRJ{\th}{\ph}
        \pgfmathparse{(\resd>=0) ? 1 : 0}
        \ifnum\pgfmathresult=1 \node[pt] at (\resx,\resy){};\else \node[bpt] at (\resx,\resy){};\fi
      }
    }
    \PRJ{0}{0}\pgfmathparse{(\resd>=0)?1:0}\ifnum\pgfmathresult=1\node[pt] at(\resx,\resy){};\else\node[bpt] at(\resx,\resy){};\fi
    \PRJ{180}{0}\pgfmathparse{(\resd>=0)?1:0}\ifnum\pgfmathresult=1\node[pt] at(\resx,\resy){};\else\node[bpt] at(\resx,\resy){};\fi

    \def\tha{30}\def\pha{55}\def\capr{20}
    \pgfmathsetmacro{\nx}{sin(\tha)*cos(\pha)}\pgfmathsetmacro{\ny}{sin(\tha)*sin(\pha)}\pgfmathsetmacro{\nz}{cos(\tha)}
    \pgfmathsetmacro{\ex}{cos(\tha)*cos(\pha)}\pgfmathsetmacro{\ey}{cos(\tha)*sin(\pha)}\pgfmathsetmacro{\ez}{-sin(\tha)}
    \pgfmathsetmacro{\fx}{\ny*\ez-\nz*\ey}\pgfmathsetmacro{\fy}{\nz*\ex-\nx*\ez}\pgfmathsetmacro{\fz}{\nx*\ey-\ny*\ex}
    \PRJV{\nx}{\ny}{\nz}\coordinate (capc) at (\vx,\vy);

    \foreach \al [count=\i from 0] in {0,5,...,355}{
      \pgfmathsetmacro{\bx}{cos(\capr)*\nx + sin(\capr)*(cos(\al)*\ex + sin(\al)*\fx)}
      \pgfmathsetmacro{\by}{cos(\capr)*\ny + sin(\capr)*(cos(\al)*\ey + sin(\al)*\fy)}
      \pgfmathsetmacro{\bz}{cos(\capr)*\nz + sin(\capr)*(cos(\al)*\ez + sin(\al)*\fz)}
      \PRJV{\bx}{\by}{\bz}
      \coordinate (B\i) at (\vx,\vy);
      \global\expandafter\edef\csname Bd\i\endcsname{\vd}
    }
    \draw[northwestern-purple!22,fill=northwestern-purple!16,line width=0pt]
      (B0) \foreach \i in {1,...,71}{ -- (B\i) } -- cycle;
    \foreach \i in {0,...,71}{
      \pgfmathtruncatemacro{\j}{Mod(\i+1,72)}
      \pgfmathparse{(\csname Bd\i\endcsname>=0)?1:0}
      \ifnum\pgfmathresult=1
         \draw[northwestern-purple!80,line width=0.6pt] (B\i)--(B\j);
      \else
         \draw[northwestern-purple!55,line width=0.5pt,dashed,dash pattern=on 1.2pt off 1.2pt] (B\i)--(B\j);
      \fi
    }
    \draw[gray!75,thin] (0,0) -- (B6);
    \draw[gray!75,thin] (0,0) -- (B42);
    \pgfmathsetmacro{\alA}{30}\pgfmathsetmacro{\alB}{210}
    \pgfmathsetmacro{\dax}{cos(\capr)*\nx + sin(\capr)*(cos(\alA)*\ex + sin(\alA)*\fx)}
    \pgfmathsetmacro{\day}{cos(\capr)*\ny + sin(\capr)*(cos(\alA)*\ey + sin(\alA)*\fy)}
    \pgfmathsetmacro{\daz}{cos(\capr)*\nz + sin(\capr)*(cos(\alA)*\ez + sin(\alA)*\fz)}
    \pgfmathsetmacro{\dbx}{cos(\capr)*\nx + sin(\capr)*(cos(\alB)*\ex + sin(\alB)*\fx)}
    \pgfmathsetmacro{\dby}{cos(\capr)*\ny + sin(\capr)*(cos(\alB)*\ey + sin(\alB)*\fy)}
    \pgfmathsetmacro{\dbz}{cos(\capr)*\nz + sin(\capr)*(cos(\alB)*\ez + sin(\alB)*\fz)}
    \PRJV{0.34*\dax}{0.34*\day}{0.34*\daz}\coordinate (arcA) at (\vx,\vy);
    \PRJV{0.34*\dbx}{0.34*\dby}{0.34*\dbz}\coordinate (arcB) at (\vx,\vy);
    \draw[gray!75,thin] (arcA) to[bend right=20] (arcB);
    \PRJV{0.60*\nx}{0.60*\ny}{0.60*\nz}
    \node[gray!95,font=\scriptsize,left] at (\vx-0.02,\vy+0.02) {$\arccos\tau$};
    \node[circle,fill=northwestern-purple,inner sep=1.0pt] at (capc){};
    \draw[->,thick] (0,0) -- (capc);
    \PRJV{\nx}{\ny}{\nz}\node[right,font=\footnotesize] at (0.62*\vx+0.04,0.55*\vy-0.06){$\x_i$};
    \draw[->,northwestern-purple,thick] ([shift={(0.04,0.06)}]capc) -- (0.80,1.20);
    \node[northwestern-purple,right,font=\footnotesize] at (0.78,1.18){cap $\B^3(\x_i,\tau)$};

    \end{tikzpicture}
    \caption{The grid $\mathcal{H}_0^3$ on the unit sphere $\mathbb{B}^3$. Grid points sit at the intersections of the latitude circles and longitude meridians; solid dots are on the visible side and hollow dots are on the back. 
    One representative point $\x_i$ is shown together with its spherical cap $\B^3(\x_i,\tau)$ of angular radius $\arccos\tau$ (shaded). Every point of $\mathcal{H}_0^3$ is the center of such a cap, and together these caps cover the sphere, so $\mathcal{H}_0^3$ is a $\tau$-hitting set.}
    \label{fig:grid-on-sphere}
    \end{figure}
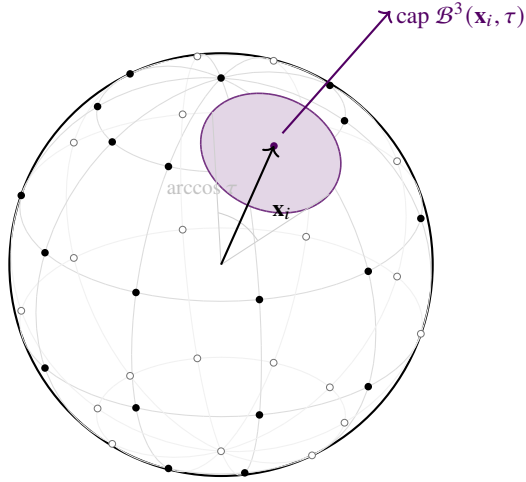
    Let $\mathbb{D}_1=\{\frac{k\pi}{m}: k=0,1,\dots,m\}$ and $\mathbb{D}_2=\{\frac{k\pi}{m}: k=0,1,\dots,2m-1\}$, then the grid points on the sphere are 
    $
        \Upsilon_0^n := \{\x\in\S^n: \varphi_1,\varphi_2,\dots,\varphi_{n-2}\in\mathbb{D}_1,\varphi_{n-1}\in\mathbb{D}_2\}.
    $
    Then $\Upsilon_0^n$ is a $\left(1-\frac{\pi^2(n-1)}{8m^2}\right)$-hitting set of size at most $2m(m+1)^{n-2}$~\cite[Lemma 3.1]{hu2025complexity}.  
    Taking $m$ sufficiently large, we can construct a $\tau$-hitting set for any $\tau\in(0,1)$. In fact, for any fixed $\tau > 0$, there does not exist a construction of $\tau$-hitting set whose size is bounded by $O(n_1^\alpha)$ with some constant $\alpha>1$ as $n_1$ gets large; see \cite[Table 1]{he2023approximation}.
    \hfill $\triangleleft$
\end{example}
Recall that $\Upsilon$ is used in \cref{prob:primal-linear-constr-discret}, which is an approximation of \cref{prob:GenLC-finite-rep,prob:tensor-comp-linear-constr}.
In the following theorem, we show that solving \cref{prob:primal-linear-constr-discret} gives us a solution with guaranteed approximation ratio to \cref{prob:tensor-comp-linear-constr} when $\Upsilon$ is constructed as a $\tau$-hitting set.
The proof of \cref{thm:approximation-ratio} is deferred to \cref{app:approximation-ratio}.

\begin{theorem}\label{thm:approximation-ratio}
    Let $\T^*$ be an optimal solution of the tensor recovery problem~\cref{prob:tensor-comp-linear-constr}. 
    Let $\Upsilon$ be a $\tau$-hitting set, and $(\widehat{\T},\{\widehat{\bm\Lambda}_i\}_{i=1}^m)$ be an optimal solution of \cref{prob:primal-linear-constr-discret} and $\widehat{V}$ be the optimal value of \cref{prob:primal-linear-constr-discret}. 
    Then it holds that 
    \begin{align}
       \tau \widehat{V} \leq \|\widehat{\T}\|_*\le \widehat{V}\label{ineq:app-1}
    \end{align}
    and 
    \begin{align}
         \tau \|\widehat{\T}\|_* \leq \|\T^*\|_*\le \|\widehat{\T}\|_*.\label{ineq:app-2}
    \end{align}
\end{theorem}

\begin{remark}[Approximate nuclear decomposition]
\label{rem:approx-nuclear-decomposition}
Recall that a nuclear decomposition of $\T$, as in
\cref{def:nuclear-decomp}, is an exact rank-one decomposition
with unit factors whose sum of absolute coefficients equals
$\|\T\|_*$.
For $\tau\in(0,1]$, we call an exact rank-one decomposition
with unit factors a \emph{$\tau$-approximate nuclear decomposition}
if its coefficients satisfy
$  \|\T\|_*
    \leq \sum_{r=1}^{R}|\lambda_r|
    \leq \frac{1}{\tau}\|\T\|_*$.

The optimal SDP solution in \cref{thm:approximation-ratio}
provides such a decomposition of $\widehat{\T}$.
Specifically, let $\widehat{\bm\Gamma}_i$ be the off-diagonal
block of $\widehat{\bm\Lambda}_i$, and let
$    \widehat{\bm\Gamma}_i
    = \sum_{r=1}^{R_i}\sigma_{ir}\y_{ir}\z_{ir}^{\top},
    ~\text{where}~
    R_i:=\operatorname{rank}(\widehat{\bm\Gamma}_i)$
be the singular value decomposition of $\widehat{\bm\Gamma}_i$,
where $\sigma_{ir}>0$, $\y_{ir}\in\R^{n_2}$,
$\z_{ir}\in\R^{n_3}$, and $\|\y_{ir}\|=\|\z_{ir}\|=1$.
Then the second constraint in
\cref{prob:primal-linear-constr-discret} can be written as
\begin{equation}\label{eq:approx-nuclear-decomposition}
    \widehat{\T}
    = \sum_{i=1}^{m}\x_i\otimes\widehat{\bm\Gamma}_i
    = \sum_{i=1}^{m}\sum_{r=1}^{R_i}
      \sigma_{ir}\x_i\otimes\y_{ir}\otimes\z_{ir}.
\end{equation}

By \cref{lemma:matrix-nuclear-norm} and optimality of the
diagonal blocks, the sum of its coefficients satisfies
\[
    \sum_{i=1}^{m}\sum_{r=1}^{R_i}\sigma_{ir}
    = \sum_{i=1}^{m}\|\widehat{\bm\Gamma}_i\|_*
    = \frac{1}{2}\sum_{i=1}^{m}\tr(\widehat{\bm\Lambda}_i)
    = \widehat{V}.
\]
Consequently, \cref{ineq:app-1} yields
$  \|\widehat{\T}\|_*
    \leq \sum_{i=1}^{m}\sum_{r=1}^{R_i}\sigma_{ir}
    \leq \frac{1}{\tau}\|\widehat{\T}\|_*$.
Hence, \cref{eq:approx-nuclear-decomposition} is a
$\tau$-approximate nuclear decomposition of $\widehat{\T}$.
It is a nuclear decomposition whenever
$\widehat{V}=\|\widehat{\T}\|_*$.
\end{remark}

To summarize, we develop an approach that gives a $\tau$-approximation solution to \cref{prob:tensor-comp-linear-constr}. The key step in this approach lies in using a $\tau$-hitting set containing finite elements to approximate the unit sphere, and the approximation ratio depends on how coarse our approximation to the unit sphere is. In the next section, we sharpen the constant approximation solution to an asymptotic optimal solution.
\begin{remark}
     We extend the approximation bound, from purely computing the nuclear norm in \cite{he2023approximation}, to that with optimization subject to general linear constraints.
     It greatly broadens the application domain of this methodology.
\end{remark}

\section{An adaptive
algorithm with asymptotic optimality}\label{sec:dual}
In \cref{sec:compute}, we used the finite representation of the tensor nuclear norm to reformulate~\cref{prob:tensor-comp-linear-constr} as a minimization problem, and obtained a $\tau$-approximation by discretizing the unit sphere with a $\tau$-hitting set fixed in advance.
A limitation of that approach is that the quality of the solution is capped by $\tau$. As $\tau\to 1$, the size of $\tau$-hitting grows exponentially, making the approach impractical. 
In this section, we propose an adaptive algorithm that asymptotically converges to the optimal solution.
This is achieved by looking at the dual problem of \cref{prob:tensor-comp-linear-constr}.

\subsection{Dual reformulation}\label{subsec:dual-reformulate}
We begin by rewriting the tensor recovery problem~\cref{prob:tensor-comp-linear-constr} as a minimax problem.
Recall that the tensor nuclear norm is the dual norm of the tensor spectral norm, i.e., $\|\T\|_*=\max_{\|\Z\|_\sigma\leq 1}\la\T,\Z\ra$.
Substituting this relation into the objective of~\cref{prob:tensor-comp-linear-constr} gives the minimax formulation
\begin{equation}\label{prob:tensor-comp-linear-constr-1}
    \min_{\T}\max_{\Z}\;\left\{\la \T,\Z\ra
    \quad\text{s.t.}\quad
    \mathscr{A}(\T)=\mathbf{b},\quad \|\Z\|_\sigma\leq 1.\right\}
\end{equation}

Problem \cref{prob:tensor-comp-linear-constr-1} is equivalent to \cref{prob:tensor-comp-linear-constr}, and \cref{prob:tensor-comp-linear-constr-1} can be viewed as the primal problem.
To derive the dual problem,
we exchange the order of the min and the max in~\cref{prob:tensor-comp-linear-constr-1}, which gives the following max-min problem
\begin{equation}\label{prob:tensor-comp-maxmin}
    \max_{\Z}\min_{\T}\;\left\{\la \T,\Z\ra
    \quad\text{s.t.}\quad
    \mathscr{A}(\T)=\mathbf{b},\quad \|\Z\|_\sigma\leq 1.\right\}
\end{equation}
We claim that $\mathrm{val}\cref{prob:tensor-comp-linear-constr-1}=\mathrm{val}\cref{prob:tensor-comp-maxmin}$, i.e., strong duality holds for \cref{prob:tensor-comp-linear-constr-1} and \cref{prob:tensor-comp-maxmin}.
To see this, note that (i) the objective $\la \T,\Z\ra$ is bilinear, hence convex in $\T$ and concave in $\Z$; (ii) the feasible set of $\T$ is $\{\T\in\R^{n_1\times n_2\times n_3}:\mathscr{A}(\T)=\mathbf{b}\}$, which is convex, and the feasible set of $\Z$ is $\{\Z\in\R^{n_1\times n_2\times n_3}:\|\Z\|_\sigma\leq 1\}$, which is convex and compact.
Using Sion's minimax theorem \citep{ref:sion1958general}, the equivalence holds.

To facilitate the analysis, we now reformulate~\cref{prob:tensor-comp-maxmin} as a single-stage maximization problem.
To do this, we derive dual problem of the inner minimization problem in \cref{prob:tensor-comp-maxmin}.
For a fixed $\Z$, the inner problem $\min_{\T}\{\la\T,\Z\ra:\mathscr{A}(\T)=\mathbf{b}\}$ is a linear program in $\T$, and its Lagrangian dual is
\[
\max_{\bm\lambda\in\R^d}
\left\{\bm\lambda^\top\mathbf{b}:
\Z=\sum_{\ell=1}^d \bm\lambda_\ell\A_\ell\right\}.
\]
where $\bm\lambda_\ell$ denotes the $\ell$th entry of $\bm\lambda$ and we use $[\mathscr{A}(\T)]_\ell=\la\A_\ell,\T\ra$.
Because the inner problem is a linear program, strong duality holds. 
So we can replace the inner minimization in~\cref{prob:tensor-comp-maxmin} by its dual.
The resulting problem shares the same optimal value and the same optimal solution as~\cref{prob:tensor-comp-maxmin}:
\[
\max_{\bm\lambda\in\R^d,\,\Z\in\R^{n_1\times n_2\times n_3}}
\left\{\bm\lambda^\top\mathbf{b}:
\|\Z\|_\sigma\leq 1,\quad
\Z=\sum_{\ell=1}^d \bm\lambda_\ell\A_\ell\right\}.
\]
For the spectral-norm constraint $\|\Z\|_\sigma\leq 1$, we use the Schur complement to replace it with a family of positive semi-definite constraints, one for each $\x\in\mathbb{B}^{n_1}$.
This gives the following one stage maximization problem :
{
\renewcommand{\theequation}{DUAL}
\begin{subequations}\label{prob:tensor-comp-dual}
\renewcommand{\theequation}{\theparentequation-\alph{equation}}
\begin{align}
\max_{\bm\lambda\in\R^d,~\Z\in\R^{n_1\times n_2\times n_3}}\;\;&
\bm\lambda^\top\mathbf{b} \\
\text{s.t.}\;\;&
\begin{bmatrix}
\I_{n_2} &\Z(\x)\\
\Z(\x)^\top &\I_{n_3}
\end{bmatrix}\succeq \bm 0
~~\forall\,\x\in\mathbb{B}^{n_1},
\label{prob:tensor-comp-dual-sdp}\\
&\Z=\sum_{\ell=1}^d \bm\lambda_\ell\A_\ell.
\label{prob:tensor-comp-dual-lincon}
\end{align}
\end{subequations}
}
Problem~\cref{prob:tensor-comp-dual} is a semi-definite program with infinitely many constraints, each corresponding to a point $\x$ on the unit sphere $\mathbb{B}^{n_1}$.
By construction, it is a reformulation of~\cref{prob:tensor-comp-maxmin} and hence of the original problem~\cref{prob:tensor-comp-linear-constr}, i.e., val\cref{prob:tensor-comp-linear-constr}=val\cref{prob:tensor-comp-maxmin}=val\cref{prob:tensor-comp-dual}.
The difficulty of solving ~\cref{prob:tensor-comp-dual} lies in the infinite family of semi-definite constraints~\cref{prob:tensor-comp-dual-sdp}.

Similar to the previous section, we choose finite number of points to approximate the unit sphere.
The difference is that we select points in $\Upsilon$ using an algorithm instead of giving by heuristics.
When the unit sphere $\mathbb{B}^{n_1}$ is approximated by a finite subset $\Upsilon=\{\x_1,\dots,\x_m\}\subseteq\mathbb{B}^{n_1}$, the problem is
{
\renewcommand{\theequation}{DUAL(D)}
\begin{subequations}
\label{prob:tensor-comp-dual-app}
\renewcommand{\theequation}{\theparentequation-\alph{equation}}
\begin{align}
    \max_{\bm\lambda\in\R^d,~\Z\in\R^{n_1\times n_2\times n_3}}\;\;&\bm\lambda^\top\mathbf{b}\\
    \text{s.t.}\;\;&\begin{bmatrix}
        \I_{n_2} &\Z(\x)\\
        \Z(\x)^\top &\I_{n_3}
    \end{bmatrix}\succeq \bm 0~~\forall\,\x\in\Upsilon,\label{constr:dual-app-sdp}\\
    &\Z=\sum_{\ell=1}^d \bm\lambda_\ell\A_\ell.
\end{align}
\end{subequations}
}
In the next proposition, we first show that \cref{prob:tensor-comp-dual-app} is well-defined.
The proof of \cref{prop:sol-dual-app} is given in \cref{app:sol-dual-app}.
\begin{proposition}\label{prop:sol-dual-app}
    If $\operatorname{span}(\Upsilon) = \R^{n_1}$, then 
    \begin{enumerate}[label=(\roman*)]
        \item the feasible set of \cref{prob:tensor-comp-dual-app} is compact; 
        \item \cref{prob:tensor-comp-dual-app} always admits an optimal solution.
    \end{enumerate}
    
\end{proposition}

In the rest of this section, we study \cref{prob:tensor-comp-dual-app} and use its solution as an approximate solution to \cref{prob:tensor-comp-dual}.

\subsection{Relation to the primal perspective}\label{subsec:primal-dual-relation}
Before proposing our algorithm, we investigate the relationship of dual problem \cref{prob:tensor-comp-dual} and the primal problem \cref{prob:GenLC-finite-rep}, and the relationship between their discretization approximation \cref{prob:primal-linear-constr-discret,prob:tensor-comp-dual-app} as a preparation.
Understanding their relations can help us to design algorithm and analyze its performance in \cref{subsec:dual-adaptive}.

First note that both reformulations are exact, i.e.:
\begin{align}\label{eq:val-chain}
    \mathrm{val}\cref{prob:tensor-comp-linear-constr}
    =\mathrm{val}\cref{prob:GenLC-finite-rep}
    =\mathrm{val}\cref{prob:tensor-comp-dual},
\end{align}
where the first equality follows from \cref{thm:nuc-norm-finite-rep} and the second follows from Sion's minimax theorem.
A common feature of the two exact formulations~\cref{prob:GenLC-finite-rep} and~\cref{prob:tensor-comp-dual} is that both involve finding points on the unit sphere $\mathbb{B}^{n_1}$: in the primal~\cref{prob:GenLC-finite-rep} the support set $\Theta\subseteq\mathbb{B}^{n_1}$ ranges over the sphere; and in the dual~\cref{prob:tensor-comp-dual} the semi-definite constraint~\cref{prob:tensor-comp-dual-sdp} is imposed at every $\x\in\mathbb{B}^{n_1}$.
To handle infinity, both perspectives use the strategy to replace the sphere by a finite discretization set $\Upsilon$, yielding
problems~\cref{prob:primal-linear-constr-discret} and ~\cref{prob:tensor-comp-dual-app}.
The following theorem gives the relation between \cref{prob:primal-linear-constr-discret} and ~\cref{prob:tensor-comp-dual-app}.

\begin{theorem}\label{thm:dual-dis}
The relation between~\cref{prob:tensor-comp-dual-app} and \cref{prob:primal-linear-constr-discret} is as follows.
\begin{enumerate}[label=(\roman*)]
    \item Problem~\cref{prob:tensor-comp-dual-app} is the Lagrangian dual problem of~\cref{prob:primal-linear-constr-discret}, and strong duality holds between them.
    \item For \cref{prob:tensor-comp-dual-app} with a finite set $\Upsilon$, $\mathrm{val}\cref{prob:tensor-comp-dual}=\mathrm{val}\cref{prob:tensor-comp-dual-app}$ if and only if there exists an optimal solution $(\T^*,\Theta^*,\bm\Lambda^*)$ to \cref{prob:GenLC-finite-rep} such that $\Upsilon=\Theta^*$.
\end{enumerate}
\end{theorem}

The proof of \cref{thm:dual-dis} is in \cref{app:dual-dis}.
\cref{thm:dual-dis}(i) shows that the discretized primal~\cref{prob:primal-linear-constr-discret} and discretized dual~\cref{prob:tensor-comp-dual-app} are Lagrangian duals under the same $\Upsilon$.
\cref{thm:dual-dis}(ii) reveals that the finite relaxation \cref{prob:tensor-comp-dual-app} is exact if and only if its discretization set $\Upsilon$ is chosen as the support set $\Theta^*$ of an optimal solution to \cref{prob:GenLC-finite-rep}.
In \cref{subsec:dual-adaptive}, we give a condition for the finite set, under which the relaxed problem \cref{prob:tensor-comp-dual-app} is a tight relaxation of \cref{prob:tensor-comp-dual} (see \cref{thm:dual-app}).
We apply this condition iteratively in an algorithm to find the set.
The relationships among all the problems studied in \cref{sec:compute,sec:dual} are summarized in \cref{fig:relations}.

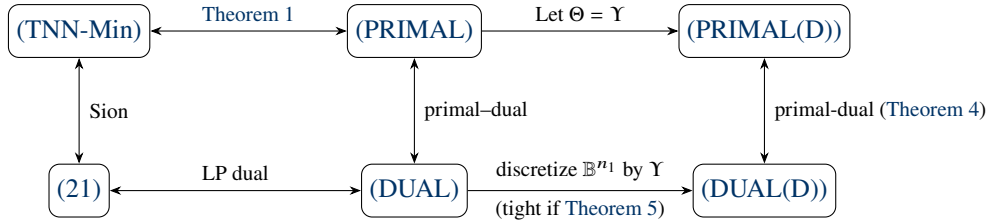
\begin{figure}[ht]
\centering
\begin{tikzpicture}[
    node distance=12mm and 26mm,
    box/.style={draw,rounded corners,inner sep=3pt,minimum height=7mm,font=\small},
    lbl/.style={font=\scriptsize,midway},
    >={Stealth[scale=0.9]}]
\node[box] (orig) {\cref{prob:tensor-comp-linear-constr}};
\node[box,right=of orig] (fin) {\cref{prob:GenLC-finite-rep}};
\node[box,right=of fin] (pdis) {\cref{prob:primal-linear-constr-discret}};
\node[box,below=14mm of orig] (maxmin) {\cref{prob:tensor-comp-maxmin}};
\node[box,below=14mm of fin] (dual) {\cref{prob:tensor-comp-dual}};
\node[box,below=14mm of pdis] (ddis) {\cref{prob:tensor-comp-dual-app}};
\draw[<->] (orig) -- node[lbl,above]{\cref{thm:nuc-norm-finite-rep}} (fin);
\draw[->] (fin) -- node[lbl,above]{Let $\Theta=\Upsilon$} (pdis);
\draw[<->] (orig) -- node[lbl,right]{Sion} (maxmin);
\draw[<->] (maxmin) -- node[lbl,above]{LP dual} (dual);
\draw[->] (dual) -- node[lbl,above]{discretize $\mathbb{B}^{n_1}$ by $\Upsilon$} node[lbl,below]{(tight if \cref{thm:dual-app})} (ddis);
\draw[<->] (fin) -- node[lbl,right]{primal--dual} (dual);
\draw[<->] (pdis) -- node[lbl,right]{primal-dual (\cref{thm:dual-dis})} (ddis);
\end{tikzpicture}
\caption{Relations among the problems. Top row: primal side; bottom row: dual side. Horizontal arrows are exact reformulations or discretizations; vertical arrows are primal--dual relationships.}
\label{fig:relations}
\end{figure}

\subsection{Adaptive approximation scheme and convergence}\label{subsec:dual-adaptive}

In this subsection, we propose an algorithm that adaptively selects points on the unit sphere to $\Upsilon$, and prove its asymptotic optimal performance. It follows a similar idea in~\cite{he2026nuclear}, but in our case we need to optimize the nuclear norm under linear constraints, making the problem harder and more general.
We solve the problem from \cref{prob:tensor-comp-dual,prob:tensor-comp-dual-app}.
First, we clarify the relation between \cref{prob:tensor-comp-dual,prob:tensor-comp-dual-app}.
For ease of notation, we let
\begin{align*}
    \Phi:=\left\{\Z:\begin{bmatrix}
        \I_{n_2} &\Z(\x)\\
        \Z(\x)^\top &\I_{n_3}
    \end{bmatrix}\succeq \bm 0,~\forall\,\x\in\mathbb{B}^{n_1}\right\}
\end{align*}
denote the feasible set induced by~\cref{prob:tensor-comp-dual-sdp}.
Also, we let
\begin{align}\label{eq:S-hat}
    \widehat{\Phi}(\Upsilon):=\left\{\Z:\begin{bmatrix}
        \I_{n_2} &\Z(\x)\\
        \Z(\x)^\top &\I_{n_3}
    \end{bmatrix}\succeq \bm 0,~\forall\,\x\in\Upsilon\right\}
\end{align}
denote the feasible set of \cref{constr:dual-app-sdp} induced by set $\Upsilon$.
Note that we have $\Phi\subseteq\widehat{\Phi}(\Upsilon)$, and~\cref{prob:tensor-comp-dual-app} is a relaxed problem of~\cref{prob:tensor-comp-dual}.
In the following lemma, we show that this containment is strict for any finite set $\Upsilon$.
The proof of \cref{lemma:SHsubsetS} is in \cref{app:SHsubsetS}.
\begin{lemma}\label{lemma:SHsubsetS}
    $\Phi\subsetneq\widehat{\Phi}(\Upsilon)$ holds for any finite set $\Upsilon\subseteq\mathbb{B}^{n_1}$.
\end{lemma}
\cref{lemma:SHsubsetS} shows that when we enlarge $\Upsilon$, it shrinks the feasible region $\widehat{\Phi}(\Upsilon)$ toward $\Phi$. 
But for any finite $\Upsilon$, the two sets are never the same.
However, note that $\Phi=\widehat{\Phi}(\Upsilon)$ is {\it not} needed for \cref{prob:tensor-comp-dual-app} to be an exact relaxation of \cref{tensor-nuclear-norm} in \emph{value}. 
If the optimal solution of~\cref{prob:tensor-comp-dual} happens to lie in $\widehat{\Phi}(\Upsilon)$ for some $ \Upsilon$, the optimal solution of \cref{prob:tensor-comp-dual-app} is also optimal to ~\cref{prob:tensor-comp-dual}.
The key is thus to find a finite $\Upsilon$ so that $\widehat{\Phi}(\Upsilon)$ includes an optimal solution.
The next theorem proposes a condition for such $\Upsilon$.
The proof of \cref{thm:dual-app} is deferred to \cref{app:dual-app}.

\begin{theorem}\label{thm:dual-app}
Let $\Upsilon=\{\x_1,\dots,\x_m\}\subseteq\mathbb{B}^{n_1}$ and $\widehat{\Z}_{\Upsilon}\in\widehat{\Phi}(\Upsilon)$ be an optimal solution of~\cref{prob:tensor-comp-dual-app}.
Consider the maximization problem
\begin{align}\label{eq:find-x-star}
\max_{\|\x\|=1}~\big\|\widehat{\Z}_{\Upsilon}(\x)\big\|_\sigma.
\end{align}
Let $\x^*$ be an optimal solution of \cref{eq:find-x-star}.
If $\x^*\in\Upsilon$, then $\widehat{\Z}_{\Upsilon}$ is also an optimal solution of~\cref{prob:tensor-comp-dual}.
\end{theorem}

\cref{thm:dual-app} gives a condition where \cref{prob:tensor-comp-dual-app} is a tight relaxation of \cref{prob:tensor-comp-dual}.
To understand the problem \cref{eq:find-x-star} in the condition, recall that \cref{prob:tensor-comp-dual} requires the semi-definite constraint to hold for every $\x\in\mathbb{B}^{n_1}$, which by Schur complement is equivalent to $\|\widehat{\Z}(\x)\|_\sigma\leq 1$ for all $\x\in\mathbb{B}^{n_1}$.
Passing to \cref{prob:tensor-comp-dual-app}, it only keeps these constraints at the points of $\Upsilon$ and drops them everywhere else, so the solution $\widehat{\Z}$ may have $\|\widehat{\Z}(\x)\|_\sigma> 1$ at some omitted $\x\not\in\Upsilon$.
The quantity $\|\widehat{\Z}(\x)\|_\sigma$ measures the amount by which the constraint at direction $\x$ is violated, and $\x^*$ gives the direction at which the violation is largest.
In particular, $\|\widehat{\Z}(\x^*)\|_\sigma\leq 1$ certifies that \emph{no} constraint is violated. 

Though \cref{thm:dual-app} is about the property of the optimal solution, 
it also inspires our algorithm design.
For each set $\Upsilon$, we can compute the most-violated direction $\x^*=\arg\max_{\|\x\|=1}\|\widehat{\Z}(\x)\|_\sigma$.
If it is already in $\Upsilon$, the relaxation is tight and we have solved the dual exactly.
We formalize this idea in \cref{alg:adaptive-H}.

\begin{algorithm}[ht]
\caption{Adaptive scheme for tensor recovery with linear constraints}\label{alg:adaptive-H}
\begin{algorithmic}[1]
\Require Constraint tensors $\A_1,\dots,\A_d$, and vector $\mathbf{b}$.
\State Initialize $\Upsilon_0=\{\x_1,\dots,\x_{n_1}\}$ where $\Upsilon_0$ spans $\R^{n_1}$.
\For{$t=0,1,\dots$}
    \State Solve~\cref{prob:tensor-comp-dual-app} with $\Upsilon=\Upsilon_t$ and obtain an optimal solution $(\bm\lambda_t^*,\Z_t^*)$.
    \State Solve $\x_t^*\in\arg\max_{\|\x\|=1}\|\Z_t^*(\x)\|_\sigma$.
    \If{$\x_t^*\in\Upsilon_t$}
        \State $\Upsilon_*=\Upsilon_t$; \textbf{exit loop}.
    \EndIf
    \State $\Upsilon_{t+1}=\Upsilon_t\cup\{\x_t^*\}$.
\EndFor
\State Solve~\cref{prob:primal-linear-constr-discret} with $\Upsilon_*$ and obtain the optimal solution $\widehat{\T}$.
\Ensure the recovered tensor $\widehat{\T}$.
\end{algorithmic}
\end{algorithm}

A difficulty in implementing \cref{alg:adaptive-H} lies in Step 4, i.e., finding $\x_i^*\in\arg\max_{\|\x\|=1}\|\Z_i^*(\x)\|_\sigma$. Although this is still an NP-hard problem of the computation of the tensor spectral norm, there are many efficient numerical methods; see the discussion in~\cite[Section 3.5]{he2026nuclear}. In general, Step 4 is replaced by an {\it approximate} spectral-norm computation (see \cite{anandkumar2017analyzing,friedland2013best,he2010approximation,he2014probability,he2023approximation}).
Our numerical results in \cref{sec:numerical} show that \cref{alg:adaptive-H} performs well even when Step 4 is computed only approximately.

The following theorem establishes the convergence of \cref{alg:adaptive-H}.
The proof of \cref{thm:alg-converge} is in \cref{app:alg-converge}.
\begin{theorem}\label{thm:alg-converge}
Suppose that \cref{prob:tensor-comp-linear-constr} is feasible and that $\mathcal{A}_1,\dots,\mathcal{A}_d$ are linearly independent.
Let $(\bm\lambda_i^*,\Z_i^*)$ be an optimal solution of iteration $i$ in \cref{alg:adaptive-H}, and let $\Psi^*$ be an optimal solution set of~\cref{prob:tensor-comp-dual}. Assume that all $\A_\ell$ are nonzero. 
If \cref{alg:adaptive-H} terminates after finite $i^*$ iterations, then $(\bm\lambda^*_{i^*},\Z^*_{i^*})$ is an optimal solution to \cref{prob:tensor-comp-dual} and $\widehat{\T}$ is an optimal solution to \cref{prob:tensor-comp-linear-constr}. If \cref{alg:adaptive-H} does not terminate in finite iterations, then
\begin{enumerate}
    \item[(i)] any limit point of $\{(\bm\lambda_t^*,\Z_t^*)\}_{t=0}^\infty$ belongs to $\Psi^*$;
    \item[(ii)] $\lim_{t\to\infty}{\bm\lambda_t^*}^\top\mathbf{b}=\mathrm{val}\cref{prob:tensor-comp-dual}$;
    \item[(iii)] let $\T_t^*$ be an optimal solution of~\cref{prob:primal-linear-constr-discret} with $\Upsilon=\Upsilon_t$; then $\lim_{t\to\infty}\|\T_t^*\|_*=\mathrm{val}\cref{prob:tensor-comp-linear-constr}$, i.e., $\T_t^*$ is asymptotically optimal for problem~\cref{prob:tensor-comp-linear-constr}.
\end{enumerate}
\end{theorem}

\section{Numerical results}\label{sec:numerical}
In this section, we evaluate the computational performance of our proposed methods on both synthetic and real data. 
We focus on tensor completion, an important special case of  \cref{prob:tensor-comp-linear-constr}. 
Our numerical experiments have two objectives. 
First, using synthetic data, we systematically examine the recovery performance of \cref{alg:adaptive-H} across different tensor dimensions and observation rates and compare it with existing tensor completion methods. 
Second, using real highway traffic speed data, we investigate whether the proposed method can effectively recover missing entries in tensors arising from practical applications.

\subsection{Synthetic Data}\label{sec:exp-synth-data}

In this section, we evaluate the performance of \cref{alg:adaptive-H} on synthetic data. 
Specifically, we generate different instances of the tensor completion problems. We solve the problem
\[
\min\{\|\T\|_*:\T(\omega)=\M(\omega)~\forall\, \omega\in\Omega\},~\text{where}~\M\text{ is given},
\]
by \cref{alg:adaptive-H} for these instances.
To gauge the performance of our algorithm, we compare \cref{alg:adaptive-H} with multiple existing tensor completion algorithms.

\textbf{Benchmark algorithms.}
We compare the performance of \cref{alg:adaptive-H} with five existing tensor completion methods: (1) 
The sum of matrix nuclear norms approach (SMNN, \citet[Page~2]{yuan2016tensor}) approximates tensor low-rankness through nuclear-norm regularization based on matrix representations of the tensor;
(2) The sum of nuclear norms of unfoldings approach (SNN, \citet{ref:liu2012tensor}) unfolds the tensor along different modes and minimizes a weighted sum of the nuclear norms of the resulting matrices;
(3)
The Kronecker-Basis-Representation approach (KBR, \citet{8000407}) promotes tensor low-rankness and sparsity through a Kronecker-basis representation of the tensor;
(4)
The Slice Learning approach (SL, \citet{farias2019learning}) exploits structural relations across tensor slices and recovers the missing data through a slice-based learning formulation. 
(5) The t-SVD approach \citep{ref:zhang2016exact} exploits low tubal rank under the tensor singular value decomposition and performs completion using tensor nuclear norm.

\textbf{Ground-truth tensor generation.}
To generate a tensor completion instance, we construct a ground-truth tensor $\M$ and an observation set $\Omega$ containing the indices of its observed entries. 
We generate a low CP-rank tensor $\M\in\R^{n_1\times n\times n}$ by
$\M=\sum_{r=1}^{R} w_r\,\x_r\otimes\y_r\otimes\z_r$, where $w_r$ for $r\in[R]$ are independently drawn from the uniform distribution on $[0,1]$, and
$\x_r\in\R^{n_1}$ and $\y_r,\z_r\in\R^n$ are independent standard Gaussian random vectors. 
We generate the observation set $\Omega$ independently across entries: each entry of $\M$ is included in $\Omega$ with probability $p_{\text{obs}}$ and is treated as missing with probability $1-p_{\text{obs}}$.
Thus, $p_{\text{obs}}$ represents the fraction of tensor entries observed in expectation.

\textbf{Parameter specification and error evaluation.} 
To compare the algorithms under different settings, we take $n_1\in\{3,5\}$, $n\in \{10,11,\dots,30\}$ and $p_{\text{obs}}\in\{0.01,0.02,\dots,0.4\}$.
When $n_1=3$ and $n_1=5$, we let $R=5$ and $R=7$ respectively.

Let $\T^*$ be the completed tensor given by the algorithm.
We use the relative error to evaluate the performance of the above algorithms on tensor completion, which is defined by ${\|\M-\T^*\|_{\text{F}}}/{\|\M\|_{\text{F}}}$.  Under each setting $(n_1,n,R, p_{\text{obs}})$, we repeat our algorithm and the benchmark algorithms $10$ times, and compute the average relative error. The result of the experiment is shown in~\cref{fig:heatmap-1} and \cref{fig:heatmap-2}. We use a grayscale colormap to visualize the performance of the algorithm: 
each color block in the figures corresponds to an experimental setting $(n_1, n, p_{\text{obs}})$, and lighter blocks indicate a smaller average relative recovery error.

Across all methods, the relative errors tend to be large when the observation probability is small, particularly when $p_{\text{obs}}\leq 0.15$. 
As $p_{\text{obs}}$ increases, the recovery error generally decreases, since more information about the ground-truth tensor becomes available. 
The improvement of \cref{alg:adaptive-H} becomes particularly visible for larger values of $n$. 
Among the six methods, \cref{alg:adaptive-H} exhibits the largest region of light-colored cells in both figures, indicating good recovery performance over a broad range of problem instances. 
Overall, the comparison suggests that directly minimizing the tensor nuclear norm using our adaptive SDP approach can lead to substantially smaller recovery errors than the benchmark tensor completion methods.

\begin{figure}[hbtp]
    \centering
    \includegraphics[width=0.8\linewidth]{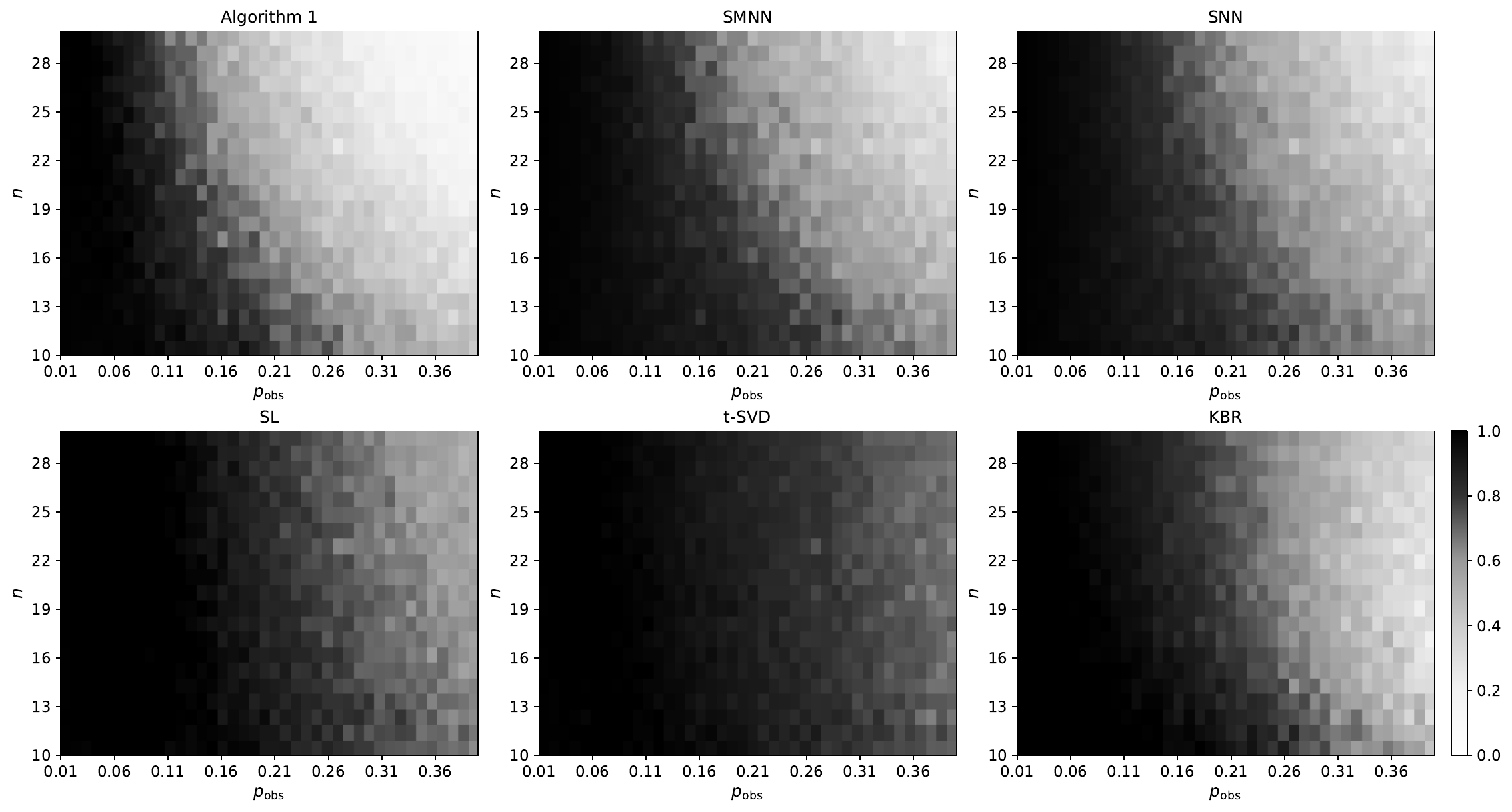}
    \caption{Grayscale colormap for all algorithms under setting $n_1=3, r=5$}
    \label{fig:heatmap-1}
\end{figure}

\begin{figure}[hbtp]
    \centering
    \includegraphics[width=0.8\linewidth]{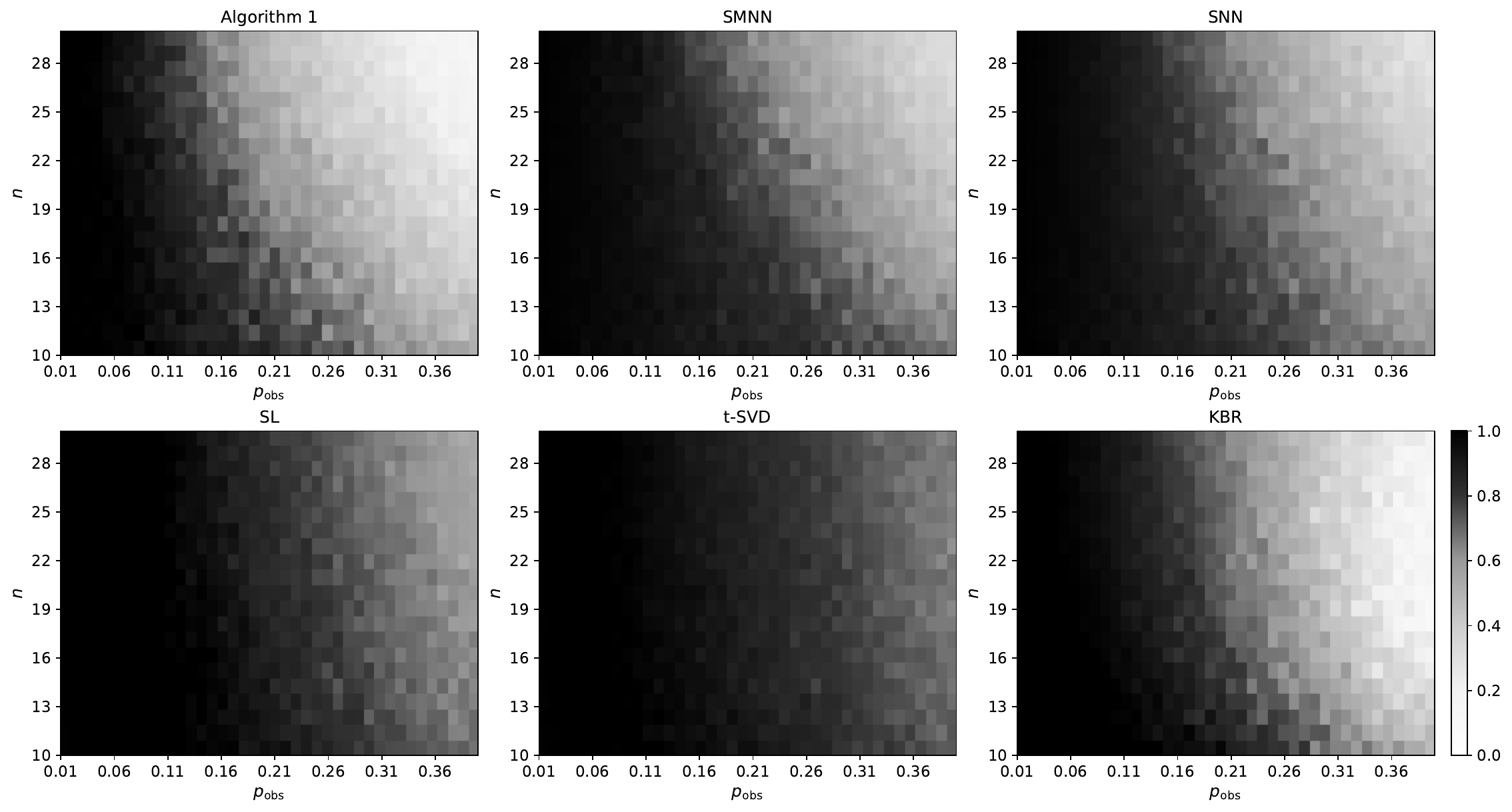}
    \caption{Grayscale colormap for all algorithms under setting $n_1=5, R=7$}
    \label{fig:heatmap-2}
\end{figure}

\subsection{Highway Traffic Speed Data}

In this section, we illustrate the performance of \cref{alg:adaptive-H} on several highway traffic speed datasets. 
We still consider the tensor completion problem.
These datasets contain traffic speed observations collected by detectors deployed on different locations on freeway networks. 
Each detector records the  vehicle speed at fixed time intervals within a day, and the measurements are collected over a sequence of days. The locations of the detectors encode the spatial information of the data
~\citep{ref:chen2021low}.

Our experiments are based on three datasets containing highway speed information at Guangzhou, Seattle and Portland (Data website: \url{https://transdim.github.io/data/}).
For each dataset, we organize it into a third-order tensor of size $D \times T \times N$, where $D$ is the number of days, $T$ is the number of time slots within a day, and $N$ is the total number of detectors. 
The $(d,t,n)$th entry of the tensor gives the traffic speed at detector $n$ during time slot $t$ on day $d$. 
In practice, traffic data have missing values due to reasons like sensor malfunction, communication errors, or data corruption. 
The traffic speed data is also known to have low-rank structure in \cite{ref:chen2021low}. 

The original dataset is large, 
with the Seattle dataset of size $28\times 288 \times 323$, the Guangzhou dataset of size $61\times 144\times 214$, and the Portland dataset of size $31\times 96\times 1156$.
According to the complexity of our algorithm, directly applying our algorithm to the full tensors will be computationally expensive.
Thus we sample a subset of the original data and recover.
Specifically, we randomly sample a subset of detectors.
Then for these selected detectors, we select a subset of days with a fixed temporal interval and choose time within a day with a fixed time interval as well. 
This is to preserve the temporal and spatial structure of the traffic data. 
Eventually, the sizes of the resampled tensors  used in the experiments are $6\times 42\times 40$ for Seattle dataset, $7\times 48\times 40$ for Guangzhou dataset, and $7\times 48\times 40$ for Portland dataset. 

For each dataset, we remove some of the elements for places of missing values and the rest are the observed set.
Let $\Omega$ denote the index set of observed data, and $\Omega$ is generated in the same way as in~\cref{sec:exp-synth-data} with $p_{\text{obs}}\in\{0.3,0.4,0.5\}$. We compare the performance of~\cref{alg:adaptive-H} with all the benchmark algorithms listed in~\cref{sec:exp-synth-data} again. Let $\T^*$ be the completed tensor given by each tensor completion algorithm.
We use the relative error to evaluate the performance of the above algorithms on tensor completion. \cref{tab:traffic-mae_p=0.3} gives the results of the experiments.
One can see that our \cref{alg:adaptive-H} achieves the lowest relative error in most of the experiments.

\begin{table}[hbtp]
\centering
\begin{tabular}{clcccccc}
\toprule
&Dataset & SMNN & \cref{alg:adaptive-H} & SNN & SL & t-SVD & KBR \\
\midrule
\multirow{3}{*}{$p=0.3$}&Seattle   & 5.9773 & \textbf{5.1074} & 5.3097 & 9.7518 & 7.8950  & 5.3932 \\
&Guangzhou   & 3.9645 & \textbf{3.6640} & 3.7173 & 6.5365 & 5.2748  & 6.0046 \\
&Portland & 22.7568 & \textbf{21.9785} & 23.9369 & 39.8669 & 97.5713 & 79.7322 \\
\midrule
\multirow{3}{*}{$p=0.4$}&Seattle   & 5.4878 & \textbf{4.6516} & 4.8327 & 8.4182 & 5.3336 & 5.2421 \\
&Guangzhou   & 3.6386 & \textbf{3.3854} & 3.3926 & 5.8308 & 4.4118 & 5.5848 \\
&Portland & 19.3457 & \textbf{18.3294} & 20.6514 & 36.9739 & 87.2549 & 74.3858 \\
\midrule
\multirow{3}{*}{$p=0.5$}&Seattle   & 5.0817 & {4.5977} & \textbf{4.4405} & 7.7203 & 4.8320 & 4.9665 \\
&Guangzhou   & 3.3781 & 3.1636 & \textbf{3.1479} & 5.1916 & 3.9205 & 5.1827 \\
&Portland & 17.4887 & \textbf{17.0099} & 18.4077 & 35.4466 & 76.0198 & 68.2486 \\
\bottomrule
\end{tabular}
\caption{Relative error on tensor completion for three traffic datasets}
\label{tab:traffic-mae_p=0.3}
\end{table}

\section{Conclusion}\label{sec:conclusion}

We study tensor nuclear norm minimization under general linear constraints and develop a semidefinite programming framework for solving this otherwise computationally difficult problem. Our analysis reveals that the semi-infinite structure induced by the tensor nuclear norm can be reformulated through finitely many directions on the unit sphere, leading to both approximation guarantees and an adaptive asymptotically optimal algorithm. 
Numerical experiments on synthetic and traffic data illustrate the practical potential of the proposed approach.

Several directions remain open. On the computational side, it would be useful to develop more efficient procedures for identifying informative sphere directions, especially for larger-scale tensors and higher-order problems. On the theoretical side, sharper bounds on the size of exact finite supports and stronger finite-convergence guarantees for adaptive schemes would further clarify the tractability of tensor nuclear norm optimization. 

\bibliographystyle{plainnat} %
\bibliography{references} 
\newpage
\appendix
\crefalias{section}{appendix}
\crefalias{subsection}{appendix}
\renewcommand{\theHsection}{A\arabic{section}}
\section{Proof of Lemmas and Theorems}\label{app:sec2}
\subsection{Proof of \texorpdfstring{\cref{lemma:nuclear-rank-ub}}{Lemma \ref*{lemma:nuclear-rank-ub}}}\label{app:nuclear-rank-ub}
\begin{proof}[Proof of \cref{lemma:nuclear-rank-ub}]
    Consider the set 
    $A := \{\x_1\otimes \x_2 \otimes \x_3: \|\x_1\|=\|\x_2\|=\|\x_3\|=1, \x_1\in \R^{n_1},\x_2\in \R^{n_2},\x_3\in \R^{n_3} \}$
    
    and define $K:= \operatorname{conv}(A)$. Then we have 
    \begin{align*}
        \|\T\|_* \;=\; &\min\left\{\sum_{r=1}^R c_r: \T = \sum_{r=1}^R c_r\mathbf{\mathcal{A}}_r, c_r \geq 0, \mathbf{\mathcal{A}}_r \in A\right\}\\
        \;=\; &\min\left\{\Bar{c}: \T = \Bar{c} \sum_{r=1}^R \frac{c_r}{\Bar{c}} \mathbf{\mathcal{A}}_r, c_r\geq 0, \mathbf{\mathcal{A}}_r\in A, \Bar{c}=\sum_{r=1}^R c_r\right\}\\
        \;=\; &\min\left\{\Bar{c}\geq 0: \T = \Bar{c} \K, \K\in K\right\}.
    \end{align*}
    where the first line is from the definition of tensor nuclear norm, and the third line is from the definition of $K$.
    \citet[Proposition 3.1]{ref:friedland2018nuclear} show that the minimizer exist for the above problem.
    Thus, for any tensor $\T\in \R^{n_1\times n_2\times n_3}$, 
    there exists $(t, \K)$ such that $t=\|\T\|_*$ and $\T=t\K$, i.e., 
    $\frac{\T}{\|\T\|_*} \in K$.
    Then by an extension of Carath\'{e}odory's theorem to connected sets \citep[Theorem 18 (ii)]{ref:Eggleston_1958}, there exist $R\leq n_1n_2n_3, c_r\geq 0, \sum_{r=1}^R c_r = 1$ and $\mathbf{\mathcal{A}}_r\in A$, such that 
    $\frac{\T}{\|\T\|_*} = \sum_{r=1}^R c_r \mathbf{\mathcal{A}}_r$.
    
    That is, 
    \begin{align}
        \label{eq:nuclear-rank-decomp}
        \T = \sum_{r=1}^R \lambda_r \mathbf{\mathcal{A}}_r, ~~\lambda_r:= \|\T\|_*c_r,
    \end{align}
    and $\sum_{r=1}^R \lambda_r = \|\T\|_*$, i.e., \cref{eq:nuclear-rank-decomp} is a nuclear decomposition. Hence, we conclude that 
    $\operatorname{rank}_*(\T) \leq n_1n_2n_3$ for any $\T\in \R^{n_1\times n_2\times n_3}$. 
\end{proof}
\subsection{Proof of \texorpdfstring{\cref{lemma:matrix-nuclear-norm}}{Lemma \ref*{lemma:matrix-nuclear-norm}}}\label{app:matrix-nuclear-norm}
\begin{proof}[Proof of \cref{lemma:matrix-nuclear-norm}]
    Without loss of generality, assume that $n_1\leq n_2$.
    To simplify the notation, we denote $\U^\star = \U^\star(\bm\Gamma), \V^\star = \V^\star(\bm\Gamma)$ in the following proof. First note that by definition of matrix nuclear norm, it is obvious that $2\|\bm\Gamma\|_* = \tr(\U^\star) + \tr(\V^\star)$. It remains to show the feasibility and optimality of $\U^\star, \V^\star$. 
    To show feasibility, let
$\bm{\Sigma}:=\mathrm{diag}(\omega_1,\dots,\omega_{n_1})$,
and let $\bQ_1\in\R^{n_1\times n_2}$ consist of the first
$n_1$ rows of $\bQ$.
Since $\D=[\bm{\Sigma}\ \bm 0]$, the SVD of $\bm\Gamma$
and the definitions in \cref{UV-def} give
\[
\bm\Gamma=\P^\top\bm{\Sigma}\bQ_1,
\qquad
\U^\star=\P^\top\bm{\Sigma}\P,
\qquad
\V^\star=\bQ_1^\top\bm{\Sigma}\bQ_1.
\]
Consequently,
\[
\begin{bmatrix}
    \U^\star & \bm\Gamma\\
    \bm\Gamma^\top & \V^\star
\end{bmatrix}
=
\begin{bmatrix}
    \P^\top\\
    \bQ_1^\top
\end{bmatrix}
\bm{\Sigma}
\begin{bmatrix}
    \P^\top\\
    \bQ_1^\top
\end{bmatrix}^{\top}
\succeq \bm 0,
\]
where the inequality follows from $\bm{\Sigma}\succeq\bm 0$.
Therefore, $\U^\star$ and $\V^\star$ are feasible.

    Now we show the optimality of $\U^\star$ and $\V^\star$. It suffices to show that for any feasible $\U$ and $\V$ to the problem~\cref{eq:inf-trace-sum}, it holds that $\tr(\U)+\tr(\V) \geq 2\|\bm\Gamma\|_*$. Let $\p_1,\dots,\p_{n_1}\in \mathbb{R}^{n_1}$ and $\q_1,\dots,\q_{n_1}\in \mathbb{R}^{n_2}$ be the left-singular vectors and right-singular vectors of $\bm\Gamma$. Let $\bm{\zeta}_j := \begin{bmatrix}
        p_j\\-q_j 
    \end{bmatrix}$ for $j=1,\dots,n_1$. Then by the semi-definiteness of $\begin{bmatrix}
        \U &\bm\Gamma\\
        \bm\Gamma^\top &\V
    \end{bmatrix}$, it holds that 
    \begin{align*}
        \bm0\leq \bm\zeta_j^\top \begin{bmatrix}
        \U &\bm\Gamma\\
        \bm\Gamma^\top &\V
    \end{bmatrix}\bm\zeta_j = \p_j^\top \U \p_j + \q_j^\top \V \q_j - \q_j^\top \bm\Gamma^\top \p_j - \p_j^\top \bm\Gamma \q_j = \p_j^\top \U \p_j + \q_j^\top \V \q_j - 2\omega_j,
    \end{align*}
    for $j=1,\dots,n_1$, where the last equality is by the definition of singular vectors.
    Then 
    \begin{align}
        \sum_{j=1}^{n_1} \p_j^\top \U \p_j + \q_j^\top \V \q_j \geq \sum_{j=1}^{n_1}2\omega_j = 2\|\bm\Gamma\|_*,\label{ineq:proof-2Gamma}
    \end{align}
    where the last equality is by the definition of matrix nuclear norm.
    Recall that we want to show that $\tr(\U)+\tr(\V) \geq 2\|\bm\Gamma\|_*$. With \cref{ineq:proof-2Gamma}, it is left to show that $\sum_{j=1}^{n_1} \p_j^\top \U \p_j + \q_j^\top \V \q_j\le  \tr(\U) + \tr(\V)$.
    To see this, we have
    \begin{align}
        \sum_{j=1}^{n_1} \p_j^\top \U 
        \p_j + 
        \q_j^\top \V \q_j \;= \;&\sum_{j=1}^{n_1} \tr(\p_j^\top \U \p_j) + \tr(\q_j^\top \V \q_j)\nonumber\\
        \;= \;&\sum_{j=1}^{n_1} \tr( \U \p_j \p_j^\top) + \tr( \V \q_j \q_j^\top)\nonumber\\
        \;=\;& \tr\left(\U\sum_{j=1}^{n_1} \p_j\p_j^\top\right)+\tr\left(\V\sum_{j=1}^{n_1} \q_j\q_j^\top\right)\nonumber\\
        \;\leq\;& \tr(\U) + \tr(\V). \label{ineq:proof-trU-trV}
    \end{align}
    In the last inequality, we use two facts: (1) $\sum_{j=1}^{n_1} \p_j\p_j^\top=\I_{n_1}$, and $\mathbf{0}\preceq \sum_{j=1}^{n_1} \q_j\q_j^\top \preceq \I_{n_2}$; (2) if $\U\succeq \bm0$ and $\bm 0\preceq \mathbf{W} \preceq \I$, then $\tr(\U\mathbf{W}) \leq \tr(\U)$.
    To see why the second statement holds, we note that 
    \begin{align*}
        \tr(\U\mathbf{W}) = \tr(\U \mathbf{Y}^\top \bm\Pi \mathbf{Y}) = \tr(\mathbf{Y}\U\mathbf{Y}^\top \bm\Pi) \leq \lambda_{\max}(\bm\Pi) \tr(\U) \leq \tr(\U),
    \end{align*}
    where $\mathbf{W}=\mathbf{Y}^\top \bm\Pi \mathbf{Y}$ is the eigenvalue decomposition of $\mathbf{W}$, and $\lambda_{\max}(\bm\Pi)$ stands for the largest eigenvalue of $\bm\Pi$ and $\lambda_{\max}(\bm\Pi)\le 1$.
    Thus, $\sum_{j=1}^{n_1} \p_j^\top \U \p_j + \q_j^\top \V \q_j\le  \tr(\U) + \tr(\V)$ holds and
    it completes the proof.
    
\end{proof}
\subsection{Proof of \texorpdfstring{\cref{thm:nuc-norm-finite-rep-2}}{Theorem \ref*{thm:nuc-norm-finite-rep-2}}.}\label{app:nuc-norm-finite-rep-2}
\begin{proof}[Proof of \cref{thm:nuc-norm-finite-rep-2}]
By \cref{lemma:nuclear-rank-ub} and the definition of tensor nuclear norm,
there exists a nuclear decomposition of $\T$,
\[
\T
=
\sum_{r=1}^{R}
\lambda_r
\x_1^{(r)}\otimes
\x_2^{(r)}\otimes
\x_3^{(r)},
\]
such that
\[
R\leq n_1n_2n_3,\qquad
\|\x_1^{(r)}\|
=
\|\x_2^{(r)}\|
=
\|\x_3^{(r)}\|
=1,
\]
and
\[
\sum_{r=1}^{R}|\lambda_r|=\|\T\|_*.
\]
Let
\[
\Theta:=\left\{
\x_1^{(1)},\dots,\x_1^{(R)}
\right\}\subseteq \mathbb{B}^{n_1}.
\]
Then $|\Theta|\leq R\leq n_1n_2n_3$.

As a relaxation of \cref{prob:nuclear-norm-SDP-formulation}, we have $\text{val}\cref{prob:nn-2}\geq \|\T\|_*$.
It remains to show the reverse inequality $\text{val}\cref{prob:nn-2}\leq \|\T\|_*$. Let $\Z$ be any feasible solution to \cref{prob:nn-2}. For every
$r=1,\dots,R$, since $\x_1^{(r)}\in\Theta$, we have
\[
\begin{bmatrix}
\I_{n_2} &
\Z(\x_1^{(r)})\\
\Z(\x_1^{(r)})^\top &
\I_{n_3}
\end{bmatrix}
\succeq 0.
\]
By the Schur complement, this is equivalent to
\[
\left\|
\Z(\x_1^{(r)})
\right\|_\sigma
\leq 1.
\]
Therefore,
\begin{align*}
\langle \T,\Z\rangle
&=
\sum_{r=1}^{R}
\lambda_r
\left\langle
\x_1^{(r)}\otimes\x_2^{(r)}\otimes\x_3^{(r)},
\Z
\right\rangle\\
&=
\sum_{r=1}^{R}
\lambda_r
\left\langle
\x_2^{(r)},
\Z(\x_1^{(r)})\x_3^{(r)}
\right\rangle\\
&\leq
\sum_{r=1}^{R}
|\lambda_r|
\left\|
\Z(\x_1^{(r)})
\right\|_\sigma\\
&\leq
\sum_{r=1}^{R}
|\lambda_r|
=
\|\T\|_*,
\end{align*}
where the first inequality follows from
$\|\x_2^{(r)}\|=\|\x_3^{(r)}\|=1$ and the definition of the
matrix spectral norm.
Since the above inequality holds for every feasible $\Z$, we have
$\text{val}\cref{prob:nn-2}\leq\|\T\|_*$, which completes the proof.

\end{proof}

\subsection{Proof of \texorpdfstring{\cref{prop:reg-primal-dis}}{Proposition \ref*{prop:reg-primal-dis}}}\label{app:reg-primal-dis}
To prove \cref{prop:reg-primal-dis}, we first prove the following lemma.
\begin{lemma}\label{lemma:rank1-decomp}
    For any $\T\in\R^{n_1\times n_2\times n_3}$ and any linearly independent set of vectors $\bm\xi_1,\dots,\bm\xi_{n_1}\in\R^{n_1}$, there exist matrices $\mathbf{Y}_1,\dots,\mathbf{Y}_{n_1}\in \R^{n_2\times n_3}$ such that 
    \[
        \T = \sum_{j=1}^{n_1} \bm\xi_i\otimes \mathbf{Y}_i.
    \]
\end{lemma}
\begin{proof}[Proof of \cref{lemma:rank1-decomp}]
    We prove by construction. Let $\bm X := [\bm\xi_1,\dots,\bm\xi_{n_1}]\in\R^{n_1\times n_1}$. 
    Then $\bm X$ is invertible. Denote $\T(j_1,j_2,j_3)$ to be the $(j_1,j_2,j_3)$th entry of $\T$ and denote $\T(\cdot, j_2,j_3):=(\T(1,j_2,j_3),\dots,\T(n_1,j_2,j_3))^\top\in\R^{n_1}$. Denote also $\mathbf{Y}_j(j_2,j_3)$ to be the $(j_2,j_3)$th entry of $\mathbf{Y}_i$. For $j\in\{1,\dots,n_1\}$, construct $\mathbf{Y}_j$ by
    \[
        \mathbf{Y}_j(j_2,j_3) = \bm\eta^{j_2,j_3}(j),~(j_2,j_3)\in [n_2]\times [n_3],
    \]
    where $ \bm\eta^{j_2,j_3}:=\bm X^{-1}\T(\cdot, j_2,j_3)\in \R^{n_1}$, 
    and $\bm\eta^{j_2,j_3}(j)$ is the $j$th entry of $\bm\eta^{j_2,j_3}$. Let 
    \[
        \T' := \sum_{j=1}^{n_1} \bm\xi_i\otimes \mathbf{Y}_j.
    \]
    Now we show that $\T = \T'$.
    Note that by definition of outer product, for $(j_2,j_3)\in\{1,\dots,n_2\}\times \{1,\dots,n_3\}$, it holds that
    \[
        \T'(\cdot,j_2,j_3) = \sum_{j=1}^{n_1} \bm\xi_i \mathbf{Y}_j(j_2,j_3) = [\bm\xi_1,\dots,\bm\xi_{n_1}] (\mathbf{Y}_1(j_2,j_3),\dots,\mathbf{Y}_{n_1}(j_2,j_3))^\top = \bm X \bm\eta^{j_2,j_3} = \T(\cdot,j_2,j_3).
    \]
    This verifies the construction and completes the proof. 
\end{proof}
\begin{proof}[Proof of \cref{prop:reg-primal-dis}]
Let $\overline{\T}$ be any tensor satisfying $\mathscr{A}(\overline{\T}) = \mathbf{b}$.
We first show that the feasible region of \cref{prob:primal-linear-constr-discret} is nonempty.
Since $\operatorname{span}(\Upsilon)=\mathbb R^{n_1}$, by \cref{lemma:rank1-decomp}, there exist matrices
$\overline{\bm\Gamma}_1,\dots,\overline{\bm\Gamma}_m\in\mathbb R^{n_2\times n_3}$ such that
\[
    \overline{\T}
    =
    \sum_{i=1}^m \x_i\otimes \overline{\bm\Gamma}_i .
\]
For each $i=1,\dots,m$, let
$s_i>0$ and define
\[
    \overline{\bm\Lambda}_i
    :=
    \begin{bmatrix}
        s_i \mathbf{I}_{n_2} & \overline{\bm\Gamma}_i \\
        \overline{\bm\Gamma}_i^\top & s_i \mathbf{I}_{n_3}
    \end{bmatrix}.
\]
By setting $s_i$ sufficiently large, $\overline{\bm\Lambda}_i\succeq 0$ for every
$i=1,\dots,m$. Therefore,
$(\overline{\T},\{\overline{\bm\Lambda}_i\}_{i=1}^m)$ is feasible to~\cref{prob:primal-linear-constr-discret}.
Thus the feasible region of~\cref{prob:primal-linear-constr-discret} is nonempty.
Moreover, since each feasible $\bm\Lambda_i$ is positive semidefinite, we have
$\operatorname{tr}(\bm\Lambda_i)\ge 0$, and hence $\text{val}\cref{prob:primal-linear-constr-discret}\geq 0$.

Now we show that an optimal solution exists. Let $\{(\T^k,\{\bm\Lambda_i^k\}_{i=1}^m)\}_{k\ge 1}$ be a convergent sequence that is feasible to~\cref{prob:primal-linear-constr-discret} and 
\[
    \lim_{i\to\infty} \frac{1}{2}\sum_{i=1}^m \tr(\bm\Lambda_i^k) = \text{val}\cref{prob:primal-linear-constr-discret}.
\]
We first show that the sequence is bounded. 
 Since the
optimal value is finite, there exists a constant $C>0$ such that, for all
sufficiently large $k$,
\[
    0\le \frac12\sum_{i=1}^m \operatorname{tr}(\bm\Lambda_i^k)\le C .
\]
Therefore, $\{\bm\Lambda_i^k\}_{k\ge 1}$ is bounded for every $i$. In particular, the sequence $\{\T^k\}_{k\ge 1}$ is also bounded, since
\[
    \T^k
    =
    \sum_{i=1}^m \x_i\otimes \bm\Gamma_i^k
\]
and $\{\bm\Gamma_i^k\}_{k\ge 1}$ is bounded for every $i$.

Now, without loss of generality, let $\lim_{k\to\infty}(\T^k,\{\bm\Lambda_i^k\}_{i=1}^m) = (\T^*, \{\bm\Lambda_i^*\}_{i=1})$. We note that $(\T^*, \{\bm\Lambda_i^*\}_{i=1})$ is also feasible to \cref{prob:primal-linear-constr-discret}, since the feasible set is closed. Finally, by the continuity of the
objective function,
\[
    \frac12\sum_{i=1}^m \operatorname{tr}(\bm\Lambda_i^*)
    = \lim_{k\to\infty} \frac{1}{2}\sum_{i=1}^m \tr(\bm\Lambda_i^k) = \text{val}\cref{prob:primal-linear-constr-discret}.
\]
Therefore, \cref{prob:primal-linear-constr-discret} admits an optimal solution, which is $(\T^*, \{\bm\Lambda_i^*\}_{i=1})$. 
\end{proof}
\subsection{Proof of \texorpdfstring{\cref{thm:approximation-ratio}}{Theorem \ref*{thm:approximation-ratio}}}\label{app:approximation-ratio}
\begin{proof}[Proof of \cref{thm:approximation-ratio}]
    Recall that the tensor nuclear norm of any tensor $\T$ can be calculated by ~\cref{prob:nuclear-norm-SDP-formulation}.
    For a given set $\Upsilon$, define a relaxation of~\cref{prob:nuclear-norm-SDP-formulation} by 
    \begin{align}\label{prob:tau-approximation-of-nuclear-norm}
        f(\T, \Upsilon)=\max_{\Z} \;\;&\la \T,\Z\ra\nonumber\\
        \text{s.t.}\;\;&\begin{bmatrix}
            \I_{n_2} &\Z(\x_i)\\
            \Z(\x_i)^\top &\I_{n_3}
        \end{bmatrix}\succeq 0\;\;\forall\, i=1,\dots,m.
    \end{align}
    The dual problem of~\cref{prob:tau-approximation-of-nuclear-norm} is given by
    \begin{align}
        \min_{\bm\Lambda_1,\dots,\bm\Lambda_m}\;\; &\sum_{i=1}^m \tr\big(\bm\Lambda_i\big) \nonumber\\
        \text{s.t.} \;\;
        &\bm\Lambda_i=\begin{bmatrix}
    \bm\Lambda^{(1)}_i &\bm\Gamma_i\\
    \bm\Gamma_i^\top&\bm\Lambda^{(2)}_i
    \end{bmatrix}\succeq 0~~i=1,\dots,m,\nonumber\\
        &\T+2\sum_{i=1}^m \x_i\otimes\bm\Gamma_i = \bm 0.\label{prob:dual-f}
    \end{align}
    Note that the Slater's condition holds for~\cref{prob:tau-approximation-of-nuclear-norm}, since $\Z=\bm 0$ is always strictly feasible, thus strong duality holds for \cref{prob:tau-approximation-of-nuclear-norm,prob:dual-f}.
    So we have that problem~\cref{prob:primal-linear-constr-discret} is equivalent to
    \begin{align}
        \min_{\T}\;\;&f(\T, \Upsilon)\nonumber\\
        \text{s.t.}\;\;&\mathscr{A}(\T)=\mathbf{b}.\label{prob:min-f}
    \end{align}
    Then by definition, $f(\widehat{\T}, \Upsilon) = \widehat{V}$. Moreover, \cite[Theorem 3.8]{he2023approximation} reveals that $f(\widehat{\T}, \Upsilon)$ is a $\tau$-approximation of tensor nuclear norm, i.e., 
    \begin{equation}\label{thm3-ineq-1}
         \tau \widehat{V}=\tau f(\widehat{\T}, \Upsilon) \leq \|\widehat{\T}\|_*\le f(\widehat{\T}, \Upsilon)=\widehat{V}
    \end{equation}
    This completes the proof of~\cref{ineq:app-1}.

    Now we prove inequality~\cref{ineq:app-2}. As $\T^*$ is the optimal solution of \cref{prob:tensor-comp-linear-constr} and $\widehat{\T}$ is feasible for~\cref{prob:tensor-comp-linear-constr}, we conclude that the first inequality holds, i.e., we have $\|\T^*\|_* \leq  \|\widehat{\T}\|_*$.
    The other side can be proven as follows. Since $\widehat{\T}$ is optimal for~\cref{prob:min-f} and $\T^*$ is feasible for \cref{prob:min-f}, we have $f(\widehat{\T},\Upsilon) \leq f({\T^*},\Upsilon)$. Then it holds that
    $$
        \tau\|\widehat{\T}\|_* \overset{\cref{thm3-ineq-1}}{\le} \tau f(\widehat{\T},\Upsilon)\leq \tau f(\T^*, \Upsilon)\leq \|\T^*\|_*
    $$
    The last inequality is also by~\cite[Theorem 3.8]{he2023approximation}. Thus we have $\tau\|\widehat{\T}\|_*\leq \|\T^*\|_*$, and it completes the proof.
    
\end{proof}
\subsection{Proof of \texorpdfstring{\cref{prop:sol-dual-app}}{Proposition \ref*{prop:sol-dual-app}}}\label{app:sol-dual-app}
\begin{proof}[Proof of \cref{prop:sol-dual-app}]

    Note that once we prove claim (i), claim (ii) is implied, because a linear function attains its maximum over a compact set.
    Now we prove (i). To show compactness, it suffices to show closeness and boundedness. Note that the closedness of the feasible region is obvious given the constraint form. It is left to show the boundedness. Without loss of generality, we assume that $\Upsilon=\{\x_1,\dots,\x_{n_1}\}$ and $\operatorname{span}(\Upsilon) = \R^{n_1}.$
    We first show that
    \begin{equation}\label{eq:C0-bound}
    \|\mathcal{Z}\|_\sigma\;\le\; C_0:= \sqrt{n_1}\,
    \|\mathbf{X}^{-1}\|_\sigma~\text{for all $\Z$ satisfying \cref{constr:dual-app-sdp},}
    \end{equation}
    where matrix $\mathbf{X}:=[\mathbf{x}_1,\dots,\mathbf{x}_{n_1}]$
    collects the points in $\Upsilon$. 
    Note that $\bm\lambda$ is the image of a continuous mapping of $\Z$, therefore, once we prove \cref{eq:C0-bound}, the boundedness of $\bm\lambda$ is also proven.
    We prove \cref{eq:C0-bound} by showing that 
    \begin{equation}\label{eq:C0-bound-1}
        \|\Z(\x)\|_\sigma \leq \sqrt{n_1}\|\mathbf{X}^{-1}\|_\sigma~~\forall\, \x\in\mathbb{B}^{n_1}.
    \end{equation} 
    For any $\x_i\in \Upsilon$, via
    the Schur complement \cref{eq:schur-spec}, \cref{constr:dual-app-sdp} is equivalent to
    $\|\mathcal{Z}(\mathbf{x}_i)\|_\sigma\le 1$. Note that every $\mathbf{x}\in\mathbb{B}^{n_1}$ can be written as    $\mathbf{x}=\sum_{j=1}^{n_1}\alpha_{j}\mathbf{x}_{j}$ with
    $\boldsymbol{\alpha}=\mathbf{X}^{-1}\mathbf{x}$. Since the
    mapping $\x\mapsto \mathcal{Z}(\,\x\,)$ is linear in $\x$, the triangle inequality gives
    \[
    \|\mathcal{Z}(\mathbf{x})\|_\sigma
    \le \sum_{j=1}^{n_1}|\alpha_j|\,
       \|\mathcal{Z}(\mathbf{x}_j)\|_\sigma
    \le \|\boldsymbol{\alpha}\|_1
    \le \sqrt{n_1}\,\|\boldsymbol{\alpha}\|_2
    \le \sqrt{n_1}\,\|\mathbf{X}^{-1}\|_\sigma,
    \]
    where the second inequality is by $\|\mathcal{Z}(\mathbf{x}_j)\|_\sigma\le 1$, and the last inequality is by $\|\boldsymbol{\alpha}\|_2 = \|\mathbf{X}^{-1}\x\|_2\leq \|\mathbf{X}^{-1}\|_{\sigma}\|\x\|_2=\|\mathbf{X}^{-1}\|$.
    Therefore, \cref{eq:C0-bound-1} holds and \cref{eq:C0-bound} follows. 
\end{proof}
\subsection{Proof of \texorpdfstring{\cref{thm:dual-dis}}{Theorem \ref*{thm:dual-dis}}}\label{app:dual-dis}
To prove \cref{thm:dual-dis}, we first prove the following lemma.
\begin{lemma}\label{lemma:ZT}
    Assume that $\mathcal{A}_1,\dots,\mathcal{A}_d$ are linearly independent.
    Let $\mathbf{G}\in\mathbb{R}^{d\times d}$ be the Gram matrix with entries $\mathbf{G}_{\ell\ell'} =\langle\mathcal{A}_\ell,\mathcal{A}_{\ell'}\rangle$.
    For any $\T$ satisfying $\mathscr{A}(\T) = b$ and $(\bm\lambda, \mathcal{Z})$ satisfying $\Z=\sum_{\ell=1}^d\bm\lambda_\ell\A_\ell$, it holds that 
    \begin{equation}
        \label{eq:ZT}
        \boldsymbol{\lambda}^\top\mathbf{b}=\langle \mathcal{Z},\mathcal{T}\rangle.
    \end{equation}
    Moreover, for any $(\bm\lambda, \mathcal{Z})$ satisfying $\Z=\sum_{\ell=1}^d\bm\lambda_\ell\A_\ell$,
    \begin{equation}\label{eq:G-inv}
        \bm\lambda = \mathbf{G}^{-1} \begin{pmatrix}
            \la \mathcal{A}_1, \mathcal{Z}\ra\\
            \la \mathcal{A}_2, \mathcal{Z}\ra\\
            \vdots\\
            \la \mathcal{A}_d, \mathcal{Z}\ra
        \end{pmatrix}.
    \end{equation}
\end{lemma}

\begin{proof}[Proof of \cref{lemma:ZT}]
    
    By the construction of $\bm G$, we have
    \[
        \mathbf{G}\boldsymbol{\lambda} = \begin{pmatrix}
            \la \mathcal{A}_1, \sum_{\ell=1}^d\bm\lambda_\ell\A_\ell\ra\\
            \la \mathcal{A}_2, \sum_{\ell=1}^d\bm\lambda_\ell\A_\ell\ra\\
            \vdots\\
            \la \mathcal{A}_d, \sum_{\ell=1}^d\bm\lambda_\ell\A_\ell\ra
        \end{pmatrix} 
        = \begin{pmatrix}
            \la \mathcal{A}_1, \mathcal{Z}\ra\\
            \la \mathcal{A}_2, \mathcal{Z}\ra\\
            \vdots\\
            \la \mathcal{A}_d, \mathcal{Z}\ra
        \end{pmatrix}.
    \]
    Since $\mathcal{A}_1,\dots,\mathcal{A}_d$ are linearly independent,
    $\mathbf{G}\succ 0$.
    Hence \cref{eq:G-inv} holds.

    To show \cref{eq:ZT}, we note that 
        \begin{equation}\label{eq:obj-identity}
    \boldsymbol{\lambda}^\top\mathbf{b}
    =\boldsymbol{\lambda}^\top\mathscr{A}(\mathcal{T})
    =\sum_{\ell=1}^{d}\boldsymbol{\lambda}(\ell)\,
       \langle \mathcal{A}_\ell,\mathcal{T}\rangle
    =\left\langle \sum_{\ell=1}^{d}\boldsymbol{\lambda}(\ell)
       \mathcal{A}_\ell,\;\mathcal{T}\right\rangle
    =\langle \mathcal{Z},\mathcal{T}\rangle, 
    \end{equation}
    where the second equality is by the definition of $\mathscr{A} (\T)$, and the last one is by $\Z=\sum_{\ell=1}^d\bm\lambda_\ell\A_\ell$. 
    
\end{proof}
\begin{proof}[Proof of \cref{thm:dual-dis}]
\emph{(i)} We derive the dual of~\cref{prob:tensor-comp-dual-app} by a standard Lagrangian argument.
Recall that $\Upsilon=\{\x_1,\dots,\x_m\}$, and we introduce multipliers $\bm\Lambda_i=\begin{bmatrix}\bm\Lambda_i^{(1)}&\bm\Gamma_i\\ \bm\Gamma_i^\top&\bm\Lambda_i^{(2)}\end{bmatrix}\succeq\bm 0$ for the $m$ semi-definite constraints.
Dualizing these constraints, problem~\cref{prob:tensor-comp-dual-app} can be written as
\begin{align*}
    \max_{\bm\lambda,\Z}\min_{\bm\Lambda_i\succeq\bm 0}~\bm\lambda^\top\mathbf{b}+\sum_{i=1}^m\left\la\bm\Lambda_i,\begin{bmatrix}
        \I_{n_2} &\Z(\x_i)\\
        \Z(\x_i)^\top &\I_{n_3}
    \end{bmatrix}\right\ra
    \quad\text{s.t.}~~\Z=\sum_{\ell=1}^d\bm\lambda_\ell\A_\ell.
\end{align*}
Exchanging the order of the max and min, the dual problem is
\begin{align}
    \min_{\bm\Lambda_i\succeq\bm 0}\left\{\max_{\bm\lambda,\Z}~\bm\lambda^\top\mathbf{b}+\sum_{i=1}^m\left\la\bm\Lambda_i,\begin{bmatrix}
        \I_{n_2} &\Z(\x_i)\\
        \Z(\x_i)^\top &\I_{n_3}
    \end{bmatrix}\right\ra
    \quad\text{s.t.}~~\Z=\sum_{\ell=1}^d\bm\lambda_\ell\A_\ell\right\}.\label{prob:lg-dual}
\end{align}
Now we prove that \cref{prob:lg-dual} is equivalent to \cref{prob:primal-linear-constr-discret}. We do this by first analyzing the inner maximization problem.
To facilitate the analysis of the inner maximization problem, let $\T'$ be any tensor satisfying $\mathscr{A}(\T') = \mathbf{b}$ but fixed. Then we have
\begin{align} 
&\max_{\bm\lambda,\Z}~\left\{\bm\lambda^\top\mathbf{b}+\sum_{i=1}^m\left\la\bm\Lambda_i,\begin{bmatrix}
        \I_{n_2} &\Z(\x_i)\\
        \Z(\x_i)^\top &\I_{n_3}
    \end{bmatrix}\right\ra
    \quad\text{s.t.}~~\Z=\sum_{\ell=1}^d\bm\lambda_\ell\A_\ell\right\}\nonumber\\
    \overset{\cref{eq:ZT}}{=}~&\max_{\bm\lambda,\Z}~\left\{\la \Z, \T'\ra + \sum_{i=1}^m \tr(\bm\Lambda_i) + \sum_{i=1}^m 2\la \x_i\otimes \bm\Gamma_i,\Z\ra\quad\text{s.t.}~~\Z=\sum_{\ell=1}^d\bm\lambda_\ell\A_\ell\right\}\nonumber\\
    =~&\max_{\bm\lambda,\Z}~ \left\{\sum_{i=1}^m \tr(\bm\Lambda_i) + \left\la \Z, \T' + 2\sum_{i=1}^m \x_i \otimes \bm\Gamma_i\right\ra\quad\text{s.t.}~~ \Z=\sum_{\ell=1}^d\bm\lambda_\ell\A_\ell\right\}\nonumber\\
    =~&\max_{\bm\lambda}~ \left\{\sum_{i=1}^m \tr(\bm\Lambda_i) + \left\la \sum_{\ell=1}^d\bm\lambda_\ell\A_\ell, \T' + 2\sum_{i=1}^m \x_i \otimes \bm\Gamma_i\right\ra\right\}\nonumber\\
    =~&\max_{\bm\lambda}~ \left\{\sum_{i=1}^m \tr(\bm\Lambda_i) +\sum_{\ell=1}^d\bm\lambda_\ell \left\la \A_\ell, \T' + 2\sum_{i=1}^m \x_i \otimes \bm\Gamma_i\right\ra\right\}\label{prob:inner-2}
\end{align}
The first equality is by the choice of $\T'$ and \cref{lemma:ZT}. 
Since $\bm\lambda\in \R^d$ is unrestricted, the optimal value of \cref{prob:inner-2} is finite if and only if 
\[
    \left\la \mathcal{A}_\ell, \T' +2\sum_{i=1}^m \x_i \otimes \bm\Gamma_i\right\ra = 0, \quad \text{for all}~\ell=1,\dots,d.
\]
By the definition of operator $\mathscr{A}(\cdot)$, we can rewrite the relation as
\[
    \mathscr{A}\left(\T' +2\sum_{i=1}^m \x_i \otimes \bm\Gamma_i\right) = \bm 0.
\]
Because $\mathscr{A}(\T') = \mathbf{b}$, for the above to hold, we must have
\[
    \mathscr{A}\left(2\sum_{i=1}^m \x_i \otimes \bm\Gamma_i\right) =-  \mathscr{A}(\T')=-\mathbf{b}.
\]
Hence, we conclude that 
\begin{align*}
    \text{val}\cref{prob:inner-2} = \left\{\begin{matrix}
        \sum_{i=1}^m \tr(\bm\Lambda_i), &\mathscr{A}\left(2\sum_{i=1}^m \x_i \otimes \bm\Gamma_i\right) = -\mathbf{b}\\
        +\infty,&\text{o.w.}
        \end{matrix}\right.
\end{align*}
Now we look at the minimax problem in \cref{prob:lg-dual}. 
Letting $\T=-2\sum_{i=1}^m \x_i \otimes \bm\Gamma_i$, we can rewrite \cref{prob:lg-dual} as
\begin{align*}
    \min_{\bm\Lambda_i,\T}~\sum_{i=1}^m\tr(\bm\Lambda_i)\quad\text{s.t.}~~\bm\Lambda_i\succeq\bm 0,~~\T+2\sum_{i=1}^m\x_i\otimes\bm\Gamma_i=\bm 0,~~\mathscr{A}(\T)=\mathbf{b},
\end{align*}
which is exactly~\cref{prob:primal-linear-constr-discret} by rescaling the decision variables $\bm\Lambda_1,\dots,\bm\Lambda_m$.
Because $\bm\lambda=\bm 0, \Z=\bm 0$ is strictly feasible for~\cref{prob:tensor-comp-dual-app}, Slater's condition holds and strong duality follows.

\emph{(ii)}  We prove the two directions separately.
We first prove the ``if'' direction. Let
$(\T^*,\Theta^*,\bm\Lambda^*)$ be an optimal solution to
\cref{prob:GenLC-finite-rep} such that $\Upsilon=\Theta^*$.
By the construction of \cref{prob:primal-linear-constr-discret},
fixing $\Theta=\Theta^*=\Upsilon$ gives a feasible solution to
\cref{prob:primal-linear-constr-discret} with the same objective value.
Hence, $\mathrm{val}\cref{prob:primal-linear-constr-discret}
=\mathrm{val}\cref{prob:GenLC-finite-rep}.$
By part (i),
$\mathrm{val}\cref{prob:tensor-comp-dual-app}
=
\mathrm{val}\cref{prob:primal-linear-constr-discret}.$
Moreover, by \eqref{eq:val-chain},
$\mathrm{val}\cref{prob:tensor-comp-dual}
=
\mathrm{val}\cref{prob:GenLC-finite-rep}.$
Therefore,
$\mathrm{val}\cref{prob:tensor-comp-dual}
=
\mathrm{val}\cref{prob:tensor-comp-dual-app}.$

Now we prove the ``only if'' direction. Suppose that
$
\mathrm{val}\cref{prob:tensor-comp-dual}
=
\mathrm{val}\cref{prob:tensor-comp-dual-app}.$
By part (i) and \eqref{eq:val-chain}, we have
$\mathrm{val}\cref{prob:primal-linear-constr-discret}
=
\mathrm{val}\cref{prob:tensor-comp-dual-app}
=
\mathrm{val}\cref{prob:tensor-comp-dual}
=
\mathrm{val}\cref{prob:GenLC-finite-rep}.$
Let
$(\widehat{\T},\{\widehat{\bm\Lambda}_i\}_{i=1}^m)$
be an optimal solution to
\cref{prob:primal-linear-constr-discret}.
Set
$\Theta^*:=\Upsilon$
and define
$\bm\Lambda^*(\x_i):=\widehat{\bm\Lambda}_i$ for $i=1,\dots,m.$
By the construction of
\cref{prob:primal-linear-constr-discret},
$(\widehat{\T},\Theta^*,\bm\Lambda^*)$ is feasible to
\cref{prob:GenLC-finite-rep}, and its objective value is
$\mathrm{val}\cref{prob:primal-linear-constr-discret}
=
\mathrm{val}\cref{prob:GenLC-finite-rep}.$
Therefore,
$(\widehat{\T},\Theta^*,\bm\Lambda^*)$ is an optimal solution to
\cref{prob:GenLC-finite-rep}. Since $\Theta^*=\Upsilon$, the result
follows.

\end{proof}
\subsection{Proof of \texorpdfstring{\cref{lemma:SHsubsetS}}{Lemma \ref*{lemma:SHsubsetS}}}\label{app:SHsubsetS}

\begin{proof}[Proof of \cref{lemma:SHsubsetS}]
We assume that $\Upsilon$ spans $\R^{n_1}$ without loss of generality: if $\Upsilon$ does not span $\R^{n_1}$, we can add elements to obtain a spanning set $\mathcal{H}''$ with $\widehat{\Phi}(\mathcal{H}'')\subseteq\widehat{\Phi}(\Upsilon)$, so proving strictness for $\mathcal{H}''$ implies it for $\Upsilon$.
It is immediate that $\Phi\subseteq\widehat{\Phi}(\Upsilon)$.
It remains to exhibit some $\Z\in\widehat{\Phi}(\Upsilon)\setminus\Phi$.
We construct such a tensor in the form $\Z:=\mathbf{u}\otimes\mathbf{M}$ with $\mathbf{u}\in\mathbb{B}^{n_1}$ and $\mathbf{M}\in\R^{n_2\times n_3}$, and show
\begin{align}
    \|\Z(\x)\|_\sigma\leq 1\quad\forall\,\x\in\Upsilon,\label{eq:ZinhatS}\\
    \|\Z(\mathbf{u})\|_\sigma> 1,\quad\text{where }\mathbf{u}\in\mathbb{B}^{n_1}.\label{eq:ZnotinS}
\end{align}
By the Schur complement,~\cref{eq:ZinhatS} is equivalent to $\Z\in\widehat{\Phi}(\Upsilon)$ and~\cref{eq:ZnotinS} is equivalent to $\Z\notin\Phi$.
Since $\Upsilon$ contains only finitely many elements, we can find $\mathbf{u}\in\mathbb{B}^{n_1}$ such that $\mathbf{u}\notin\Upsilon$ and $-\mathbf{u}\notin\Upsilon$.
Let $c:=\max_{\x\in\Upsilon}|\la\mathbf{u},\x\ra|$; then $0<c<1$ because $\Upsilon$ spans $\R^{n_1}$ while $\mathbf{u}$ is parallel to no element of $\Upsilon$.
Let $\mathbf{M}\in\R^{n_2\times n_3}$ satisfy $\|\mathbf{M}\|_\sigma=1/c$.
Then for any $\x\in\Upsilon$,
\begin{align*}
    \|\Z(\x)\|_\sigma=\|(\mathbf{u}\otimes\mathbf{M})(\x)\|_\sigma=\|\la\mathbf{u},\x\ra\mathbf{M}\|_\sigma=|\la\mathbf{u},\x\ra|\,\|\mathbf{M}\|_\sigma\leq c\cdot\frac{1}{c}=1,
\end{align*}
while
\begin{align*}
    \|\Z(\mathbf{u})\|_\sigma=\|\la\mathbf{u},\mathbf{u}\ra\mathbf{M}\|_\sigma=\|\mathbf{M}\|_\sigma=\frac{1}{c}>1.
\end{align*}
This proves~\cref{eq:ZinhatS} and~\cref{eq:ZnotinS}, and hence $\Phi\subsetneq\widehat{\Phi}(\Upsilon)$.

\end{proof}
\subsection{Proof of \texorpdfstring{\cref{thm:dual-app}}{Theorem \ref*{thm:dual-app}}}\label{app:dual-app}
\begin{proof}[Proof of \cref{thm:dual-app}]
    For simplicity, we use $\widehat{\Z}$ as shorthand of $\widehat{\Z}_{\Upsilon}$ in this proof.
    By the assumption that $\x^*\in\Upsilon$ and $
    \widehat\Z\in\widehat{\Phi}(\Upsilon)$, we have 
    \[
        \begin{bmatrix}
        \I_{n_2} &\widehat\Z(\x^*)\\
        \widehat\Z(\x^*)^\top &\I_{n_3}
        \end{bmatrix}\succeq \Zero.
    \]
By Schur complement, it is also equivalent to 
    \[
        \I_{n_3}\succeq \widehat\Z(\x^*)^\top\widehat\Z(\x^*) \;\Leftrightarrow\;  \left\|\widehat\Z(\x^*)\right\|_\sigma \leq 1.
    \]
    As $\x^*$ is the optimal solution of $\max_{\|\x\|=1} \left\|\widehat{\Z}(\x)\right\|_{\sigma}$, we also have
    \[
        \left\|\widehat\Z(\x)\right\|_\sigma\le \left\|\widehat\Z(\x^*)\right\|_\sigma  \leq 1~~\forall\, \x\in \S^{n_1},
    \]
    which, by Schur complement, is equivalent to
    \[
        \begin{bmatrix}
        \I_{n_2} &\widehat\Z(\x)\\
        \widehat\Z(\x)^\top &\I_{n_3}
    \end{bmatrix}\succeq \Zero\;\;\forall\, \x\in \S^{n_1}.
    \]
    Thus it completes the proof. 
\end{proof}
\subsection{Proof of \texorpdfstring{\cref{thm:alg-converge}}{Theorem \ref*{thm:alg-converge}}}\label{app:alg-converge}

\begin{proof}[Proof of \cref{thm:alg-converge}]
    For $t=0,1,\dots$, let $v_t$ denote the optimal value of
    \cref{prob:tensor-comp-dual-app} with
    $\Upsilon=\Upsilon_t$, so that
    $v_t=\boldsymbol{\lambda}^{*\top}_t\mathbf{b}$, and let
    $m_t := |\Upsilon_t|$. We first analyze the case in which
    \cref{alg:adaptive-H} terminates after finitely many iterations and then we establish claim (i)--(iii) for the infinitely many iterations case.

    \textbf{Finite termination.}
    If Algorithm~\ref{alg:adaptive-H} 
    terminates after $t^*$ iterations, it means $\x_{t^*}^*\in \Upsilon_{t^*}$ and by \cref{thm:dual-app}, the optimality of $\Z_{t^*}^*$ for \cref{prob:tensor-comp-dual} follows. Furthermore, under $\Upsilon_{t^*}$ and $\widehat{\T}$, we have \[
        \text{val}\cref{prob:primal-linear-constr-discret} = \text{val}\cref{prob:tensor-comp-dual-app} = \text{val}\cref{prob:tensor-comp-dual} = \text{val}\cref{prob:GenLC-finite-rep} = \text{val}\cref{prob:tensor-comp-linear-constr},
    \]
    where the first equality is by \cref{thm:dual-dis}(i), the second one is by \cref{thm:dual-app}, the third and last one is by \cref{eq:val-chain}. Because $\widehat{\T}$ is feasible to \cref{prob:tensor-comp-linear-constr}, $\widehat{\T}$ is an optimal solution to \cref{prob:tensor-comp-linear-constr}.

    \textbf{The asymptotic convergence case: Proof of claim (i) and (ii).}
    We first show the existence of limit points of the sequences $\{\bm\lambda_t^*, \Z_t^*\}_{t\geq 0}$ and $\{\x_t^*\}_{t\geq 0}$.
    The feasible set of \cref{prob:tensor-comp-dual-app} is compact according to \cref{prop:sol-dual-app}(i). Using the Bolzano--Weierstrass theorem,
    $\{(\boldsymbol{\lambda}^*_t,\mathcal{Z}^*_t)\}_{t\ge 0}$ has at least one limit point.
    Let $\{(\boldsymbol{\lambda}^*_{t_k},\mathcal{Z}^*_{t_k})\}_{k\ge0}$ be a convergent subsequence of $\{(\boldsymbol{\lambda}^*_t,\mathcal{Z}^*_t)\}_{t\ge 0}$ and $\lim_{k\to\infty} (\boldsymbol{\lambda}^*_{t_k},\mathcal{Z}^*_{t_k}) =(\overline{\boldsymbol{\lambda}},\overline{\mathcal{Z}})$.
    We take the index $\{t_k\}_{k\geq 0}$ of $\{(\boldsymbol{\lambda}^*_{t_k},\mathcal{Z}^*_{t_k})\}_{k\ge0}$ and consider the corresponding subsequence $\{\x_{t_k}^*\}_{k\geq 0}$.
    Since $\mathbb{B}^{n_1}$ is compact, 
    there exists a subsequence of $\{\x^*_{t_k}\}$, denoted by $\{\x^*_{{t_k}_j}\}$, such that $\mathbf{x}^*_{{t_k}_j}\to\overline{\mathbf{x}}$ for some $\overline{\mathbf{x}}\in\mathbb{B}^{n_1}$. 
    Note that $\lim_{j\to\infty} (\bm\lambda^*_{{t_k}_j}, \Z^*_{{t_k}_j}) = (\overline{\bm\lambda}, \overline{\Z})$. 

    To prove claim (i) and (ii),
    we first show that $(\overline{\boldsymbol{\lambda}},\overline{\mathcal{Z}})$ is feasible to \cref{prob:tensor-comp-dual}. For the linear constraint \cref{prob:tensor-comp-dual-lincon}: $\mathcal{Z}^*_{{t_k}_j}
    =\sum_{\ell=1}^d(\boldsymbol{\lambda}^*_{{t_k}_j})_\ell\mathcal{A}_\ell$ for every
    $j$, and both sides are continuous in
    $(\boldsymbol{\lambda},\mathcal{Z})$, so
    $\overline{\mathcal{Z}}
    =\sum_{\ell}\overline{\boldsymbol{\lambda}}_\ell\mathcal{A}_\ell$, i.e.,
    \cref{prob:tensor-comp-dual-lincon} holds.
     
    Now we consider the semi-definite constraints \cref{prob:tensor-comp-dual-sdp}. 
    For all $j\geq 1$, we note that $\Z^*_{{t_k}_j}\in \widehat{\Phi}(\Upsilon_{{t_k}_j}) \subseteq \widehat{\Phi}(\Upsilon_{{t_k}_{j-1}})$. Then it holds that 
    \[
        \begin{bmatrix}
        \I_{n_2} &\Z^*_{{t_k}_j}(\x^*_{{t_k}_{j-1}})\\
        \Z^*_{{t_k}_j}(\x^*_{{t_k}_{j-1}})^\top &\I_{n_3}
    \end{bmatrix}\succeq \bm 0 ~\overset{\cref{eq:schur-spec}}{\Rightarrow}~ \|\Z^*_{{t_k}_j}(\x^*_{{t_k}_{j-1}})\|_\sigma \leq 1,
    \]
    for all $j\geq 1$. Let $j\to \infty$, we have 
    $\|\overline{\Z}(\overline{\x})\|_\sigma \leq 1$.
    
    In particular, by \cref{eq:tensor-norm-matrix-norm}, we have
$\|\mathcal{Z}^*_{{t_k}_j}\|_\sigma
    =\max_{\|\mathbf{x}\|=1} \|\mathcal{Z}^*_{{t_k}_j}(\mathbf{x})\|_\sigma
    = \|\Z^*_{{t_k}_j}(\x^*_{{t_k}_j})\|_\sigma
    \qquad\text{for every } j$ .
    
    Then 
 $\|\overline{\mathcal{Z}}\|_\sigma
    =\lim_{j\to\infty}\|\mathcal{Z}^*_{{t_k}_j}\|_\sigma
    =\lim_{j\to\infty} \|\Z^*_{{t_k}_j}(\x^*_{{t_k}_j})\|_\sigma
    = \|\overline{\Z}(\overline{\x})\|_\sigma \leq 1$.
    Hence
$\|\overline{\mathcal{Z}}\|_\sigma\le 1$. 
    Again, by  \cref{eq:tensor-norm-matrix-norm} and
    \cref{eq:schur-spec}, $\overline{\Z}$ is feasible to 
    \cref{prob:tensor-comp-dual-sdp}. Therefore $(\overline{\boldsymbol{\lambda}},\overline{\mathcal{Z}})$ is feasible to \cref{prob:tensor-comp-dual}.
     
    It remains to show that $(\overline{\boldsymbol{\lambda}},\overline{\mathcal{Z}})$ is optimal.
    To prove this argument, we first show that the limit of $\{v_t\}_{t\geq 0}$ exists and $\lim_{t\to\infty} v_{t}
    = v_\infty
    \ge\mathrm{val}\cref{prob:tensor-comp-dual}$. To see this,     
    By the construction of the algorithm, it holds that
    $\Upsilon_0\subseteq\Upsilon_1\subseteq\cdots$ and consequently, $\widehat{\Phi}(\Upsilon_{t+1})\subseteq \widehat{\Phi}(\Upsilon_t)$. Hence the feasible region of
    \cref{prob:tensor-comp-dual-app} shrinks as $t$ grows and
    $v_0\ge v_1\ge\cdots$. Moreover,
    $\Phi\subseteq\widehat{\Phi}(\Upsilon_t)$ for every
    $t$, so every feasible point of \cref{prob:tensor-comp-dual} is feasible
    for \cref{prob:tensor-comp-dual-app}, which gives
    $v_t\ge\mathrm{val}\cref{prob:tensor-comp-dual}$ for all $t$. Since $\{v_t\}_{t\ge0}$ is nonincreasing and bounded below, it converges to a limit
    $v_\infty\ge\mathrm{val}\cref{prob:tensor-comp-dual}$.

    By \cref{lemma:ZT}, for any $\T$ satisfying $\mathscr{A}(\T) = \mathbf{b}$,
    \[
    \overline{\boldsymbol{\lambda}}^\top\mathbf{b}
    =\langle\overline{\mathcal{Z}},\mathcal{T}\rangle
    =\lim_{t\to\infty}\langle\mathcal{Z}^*_{t},\mathcal{T}\rangle
    =\lim_{t\to\infty} v_{t}
    = v_\infty
    \ge\mathrm{val}\cref{prob:tensor-comp-dual}.
    \]
    Since $(\overline{\boldsymbol{\lambda}},\overline{\mathcal{Z}})$ is feasible to \cref{prob:tensor-comp-dual},     $\overline{\boldsymbol{\lambda}}^\top\mathbf{b} = v_{\infty}
    =\mathrm{val}\cref{prob:tensor-comp-dual}$ and hence,     $(\overline{\boldsymbol{\lambda}},\overline{\mathcal{Z}})\in\Psi^*$. This proves claim (i) and (ii).
     
    \textbf{Proof of claim (iii).}
    To prove the claim, we first prove $\limsup_{t\to\infty}\|\mathcal{T}^*_t\|_*\le\mathrm{val}\cref{prob:tensor-comp-linear-constr}$ and then we show $\|\mathcal{T}^*_t\|_*\ge\mathrm{val}\cref{prob:tensor-comp-linear-constr}$ for every $t$.
    To show the ``$\leq$'' direction, note that by the second inequality in \cref{ineq:app-1} and \cref{thm:dual-dis}(i),
    \[
        \|\mathcal{T}^*_t\|_*\le\mathrm{val}\cref{prob:primal-linear-constr-discret}
    =\mathrm{val}\cref{prob:tensor-comp-dual-app}=\boldsymbol{\lambda}^{*\top}_t\mathbf{b},
    \]
    so claim (ii) and \cref{eq:val-chain} yield $\limsup_{t\to\infty}\|\mathcal{T}^*_t\|_*\le\mathrm{val}\cref{prob:tensor-comp-dual}
    =\mathrm{val}\cref{prob:tensor-comp-linear-constr}$. Now we prove the ``$\geq$'' direction. Note that
    $\mathscr{A}(\mathcal{T}^*_t)=\mathbf{b}$ by the last constraint of
    \cref{prob:primal-linear-constr-discret}, so $\mathcal{T}^*_t$ is feasible for
    \cref{prob:tensor-comp-linear-constr}, therefore,
    $\|\mathcal{T}^*_t\|_*\ge\mathrm{val}\cref{prob:tensor-comp-linear-constr}$ for every $t$.
    Combining the two directions, we have $\lim_{t\to\infty}\|\mathcal{T}^*_t\|_*
    =\mathrm{val}\cref{prob:tensor-comp-linear-constr}$, i.e., $\mathcal{T}^*_t$ is
    asymptotically optimal for the tensor recovery problem. 

\end{proof}

\end{document}